\documentclass[10pt]{article}

\usepackage{amsfonts, epsfig, amsmath, amssymb, color,amsthm}
\usepackage{url}
\usepackage{textcomp}
\usepackage{paralist}
\usepackage[normalem]{ulem}

\usepackage[english]{babel}
\newcommand{\DD}{\mathbb{D}}
\newcommand{\EE}{\mathbb{E}}

\newcommand{\JJ}{\mathbb{J}}

\newcommand{\NN}{\mathbb{N}}

\newcommand{\PP}{\mathbb{P}}
\newcommand{\QQ}{\mathbb{Q}}
\newcommand{\RR}{\mathbb{R}}

\newcommand{\aA}{\mathcal{A}}
\newcommand{\bB}{\mathcal{B}}
\newcommand{\cC}{\mathcal{C}}

\newcommand{\eE}{\mathcal{E}}
\newcommand{\fF}{\mathcal{F}}
\newcommand{\gG}{\mathcal{G}}

\newcommand{\nN}{\mathcal{N}}
\newcommand{\oO}{\mathcal{O}}
\newcommand{\pP}{\mathcal{P}}

\newcommand{\rR}{\mathcal{R}}

\newcommand{\uU}{\mathcal{U}}

\newcommand{\xX}{\mathcal{X}}

\newcommand{\fA}{\mathfrak{A}}

\newcommand{\fM}{\mathfrak{M}}

\newcommand{\al}{\alpha}

\newcommand{\vt}{\vartheta}
\newcommand{\e}{\varepsilon}
\newcommand{\la}{\lambda}

\newcommand{\no}{\noindent}

\newcommand{\pd}{\partial}

\newcommand{\ra}{\rightarrow}

\newcommand{\lra}{\longrightarrow}
\newcommand{\ti}{\tilde}

\newcommand{\lgl}{\ensuremath{\langle}}
\newcommand{\rgl}{\ensuremath{\rangle}}

\newcommand{\non}{\nonumber}

\newcommand{\ind}{\mathbf{1}}

\newcommand{\lqq}{\leqslant}
\newcommand{\gqq}{\geqslant}

\newtheorem{thm}{Theorem}

\newtheorem{lem}[thm]{Lemma}
\theoremstyle{plain} 
\newtheorem{theorem}{Theorem}
\newtheorem{corollary}[thm]{Corollary}
\newtheorem{lemma}[thm]{Lemma}
\newtheorem{proposition}[thm]{Proposition} 

\theoremstyle{definition}

\newtheorem{claim}{Claim}
\newtheorem{condition}{Hypothesis}

\theoremstyle{remark} 

\theoremstyle{definition}

\newtheorem{rem}[thm]{Remark}

\DeclareMathSymbol{\ophi}{\mathalpha}{letters}{"1E}

\renewcommand{\phi}{\varphi}

\newcommand{\be}{\begin{equation}}
\newcommand{\ee}{\end{equation}}
\newcommand{\ben}{\begin{equation*}}
\newcommand{\een}{\end{equation*}}

\newcommand{\ba}{\begin{equation}\begin{aligned}}
\newcommand{\ea}{\end{aligned}\end{equation}}

\DeclareMathOperator{\diam}{diam}

\newfont{\cyrfnt}{wncyr10}
\def\J3{\cyrfnt{\rm \u{\cyrfnt I}}}
\def\j3{\cyrfnt{\rm \u{\cyrfnt i}}}

\usepackage[]{color}
\definecolor{DarkGreen}{rgb}{0.1,0.7,0.3}   

\definecolor{DarkGreen}{rgb}{0.1,0.7,0.3}   

\allowdisplaybreaks[4]

\begin{document}

\title{
The Kramers problem for SDEs driven by 
small, accelerated L\'evy noise with 
exponentially light jumps 
}

\date{\null}

\author{
Andr\'e de Oliveira Gomes \footnote{Departamento de Matem\'{a}tica Universidade Estadual de Campinas 13081-970 Campinas SP-Brazil; Institut f\"{u}r Mathematik Universit\"{a}t Potsdam; andrego@unicamp.br} \hspace{1cm}
Michael A. H\"ogele \footnote{Departamento de matem\'aticas, Universidad de los Andes, Bogot\'a, Colombia; ma.hoegele@uniandes.edu.co
}
}

\maketitle

 \begin{abstract} 
 We establish Freidlin-Wentzell results for 
 a nonlinear ordinary differential equation 
 starting close to the stable state $0$, say, subject to a perturbation 
 by a stochastic integral which is driven by an $\e$-small and $(1/\e)$-accelerated 
 L\'evy process with exponentially light jumps. 
 For this purpose we derive a large deviations principle for the stochastically perturbed system using the weak convergence approach developed by Budhiraja, Dupuis, Maroulas and collaborators in recent years.
 In the sequel we solve the associated asymptotic first escape problem from the bounded neighborhood of $0$ 
 in the limit as $\varepsilon \ra 0$ which is also known as the Kramers problem in the literature. 
 \end{abstract}

\noindent \textbf{Keywords: } 
Freidlin-Wentzell theory; Large deviations principle; 
accelerated small noise L\'evy diffusions; 
first passage times; 
first exit location; 
strongly tempered stable L\'evy measure.

\medskip

\noindent \textbf{2010 Mathematical Subject Classification: } 60H10; 60F10; 60J75; 60J05; 60E07.


\section{Introduction and main results}
\subsection{Introduction}

In this article we solve the Kramers problem 
for the family of strong solutions $(X^{\e})_{\e>0}$ of the following stochastic differential equation (SDE for short) 
\begin{align}\label{eq:introsys}
 X^{\e}_t = x + \int_0^t b(X^{\e}_s) ds + \varepsilon \int_0^t \int_{\RR^d\backslash\{0\}} G(X^{\e}_{s-})z  \tilde N^{\frac{1}{\e}}(ds,dz), 
\quad t \gqq 0,
\end{align}
on a neighborhood  $D$ of $0$, which is the stable state of the underlying deterministic 
dynamical system ($\varepsilon=0$). The driver of the stochastic perturbation  
 $\tilde N^{\frac{1}{\e}}$ is a compensated Poisson random measure with 
compensator $ds \otimes \frac{1}{\e}\nu$, where $\nu$ is a L\'evy measure satisfying a 
certain exponential integrability condition and the multiplicative coefficient $G$ being a 
non-vanishing scalar Lipschitz function. Our main result determines the asymptotic behavior 
($\varepsilon \ra 0$) of the law and the expectation of the first exit time and location,
\begin{align*}
\sigma^\varepsilon(x):= \inf \{ t \gqq 0 ~|~X^{\varepsilon,x}_t \notin D \} 
\quad \text{and } \quad X^\varepsilon_{\sigma^\varepsilon(x)}, \mbox{ respectively. }
\end{align*} 
Our analysis relies on the establishment of 
a large deviations principle (LDP for short) based on the weak convergence approach, 
developed by Budhiraja, Dupuis, Maroulas and collaborators \cite{BDM11, BCD13}.

\noindent Historically, the Kramers problem, that is, 
the escape time and location of a 
randomly excited deterministic dynamical system 
from close to a stable state at small intensity 
arose in the context of chemical reaction kinetics 
\cite{Arrhenius-89}, \cite{Eyring-35} and \cite{Kramers-40}. 
Nowadays this classical problem is virtually ubiquitous  
and provides crucial insights in many 
diverse areas ranging from statistical mechanics, statistics, 
insurance mathematics, population dynamics and fluid dynamics to neurology. 
The mathematical theory of large deviations goes back 
to the seminal work by Cr\'{a}mer \cite{Cramer-38} before taking off in the seventies of the last century
with the fundamental works by \cite{DV75, FW70, FW98, Wentzell-76}. 
One main focus was the first exit problem for 
ordinary, delay and partial differential equations with small Gaussian noise 
in different settings and effects derived from it such as metastability and stochastic resonance.
Classical texts with detailed expositions of the history of large deviations theory 
include \cite{Berglund-13, 
BerglundG-04, BerGen-10, BarBovMel10, BovEckGayKle02, Bovier1, CerRoeck-04, DeuStr89,DZ98, Dupuis Ellis, Freidlin00, Siegert, Varadhan} 
among others as well as the references therein.
Furthermore, there is a lot of active research in the field, see for instance \cite{19-3, 19-2, 19-4, 19-5, SS19, 19-1}. 
The major part of the body of literature studying large deviations and 
the Kramers law for stochastic differential equations 
with small noise is devoted to the study of Gaussian dynamics. 
For the dynamics of Markovian systems with jumps the literature is 
noticeably more fragmented, scattered and recent. 
It is due to the considerable variety of L\'evy processes, including 
processes with heavy tails, and the
resulting lack of moments, that there is no general large deviations 
theory for L\'evy processes and diffusions with jumps.
Large deviations results  for certain classes of L\'{e}vy 
noises and Poisson random measures 
are given in \cite{Acosta, Blanchet et al17, Borovkov, Florens Pham, Godovanchuk-82, Leonard, Lynch87, Puhalskii} and \cite{Wentzell-90}. 
The first exit problem for small jump L\'evy processes starts with the seminal paper by \cite{Godovanchuk-79} 
for $\alpha$-stable processes and for more generally heavy-tailed processes by \cite{Debussche et al., HoePav-13, Imk Pav06, Imk Pav08, Pav11}. The mentioned works do not follow 
a large deviations regime in the sense of \cite{BDM11, BCD13} since the intensity measure 
of the underlying L\'{e}vy process is not rescaled by $\frac{1}{\varepsilon}$ as in our work. 
 We mention \cite{Imk Pav Wetzel09, ImkPavWet10}, where the authors provide 
in one dimension a complete scale study of asymptotic exit times as functions of $1/\e$ ranging 
from polynomial via subexponential to exponential. 

\noindent This article follows the rather different line of research started in \cite{BDM11}
including not only an $\e$-dependent amplitude but also an $\e$-dependent 
acceleration of the jump intensity of the noise. It is this tuning between the jump 
size of order $\varepsilon$ and the intensity of $ N^{\frac{1}{\varepsilon}}$ that permits 
to retrieve the large deviations regime for the dynamical system perturbed by a stochastic integral with respect to $\varepsilon \tilde N^{\frac{1}{\varepsilon}}$. 
In  \cite{BDM11, BCD13} Budhiraja, Dupuis, Maroulas  and collaborators derive large deviations 
results using  new variational representation formulas for functionals of continuous time processes of this type.
This has sparked  a lot of ongoing research, cf. \cite{BN15, Yang et al15, Zhai Zhang}. 
In \cite{BDM11} and in the recent monograph \cite{Budhiraja} the reader will find an extensive 
and up-to-date introduction 
to this subject. 

\noindent The concise comprehension of the Kramers problem in this setting, which to our knowledge 
is missing in the literature to date, opens the door to deeper questions such as metastability, 
stochastic resonance and averaging, among other topics. 
The LDP of $(X^{\varepsilon,x})_{\varepsilon>0}$ is given as an optimization 
problem under the dynamics solved by continuous 
controlled paths with a nonlocal component due to the pure jump noise. 

\noindent For the derivation of the LDP we verify the sufficient abstract criteria established in \cite{BDM11}. In  their follow-up article \cite{BCD13} the authors apply this criteria to prove a large deviations result for a stochastic differential equation driven by small L\'{e}vy noise under the stricter assumption that the jumps of the noise component are bounded. There the authors use the same abstract sufficient criteria to establish a LDP for a stochastic partial differential equation driven by pure jump processes where the condition on the L\'{e}vy measure is relaxed (Condition 3.1 in \cite{BDM11}). 
In contrast to \cite{BCD13} our setting is the jump diffusion given by (1)
with values in the finite dimensional Euclidean space $\RR^d$.  
Our assumptions on the coefficients of (1) are rather standard monotonicity, Lipschitz and boundedness assumptions given in full detail in Subsection 1.2 and 1.4. 
They yield the LDP (first main theorem of this article) for a  L\'{e}vy measure with infinite intensity allowing for unbounded jumps subject to an exponential integrability condition, 
such as in \cite{BCD13}. However, the method to show the validity of the abstract 
sufficient criteria given in \cite{BDM11, BCD13} is different. 
The authors there base the construction 
of their weak convergence arguments only on a-priori bounds for the second moments of the jump infinite-dimensional jump diffusions. 
Naturally, the establishment of such a-priori bounds 
is difficult to obtain in the case of locally 
Lipschitz coefficients. The method we use to derive 
the large deviations result relies on the estimation 
of probabilities and on localization techniques based on 
a Bernstein-type inequality given in \cite{DZ01}. 
The use of localization techniques naturally allows 
the extension of the large deviations result obtained in 
this article to the case of locally Lipschitz coefficients, 
such as for instance the gradient case of a polynomial potential, 
which is clearly beyond the scope of this article. 


 \no Analogously to the classical Freidlin-Wentzell theory we solve the Kramers problem with a 
pseudo-potential given in terms of the good rate function of the LDP. In the Brownian 
case, under very mild assumptions on the coefficients of the SDE the respective controlled dynamics exhibits continuity properties that are crucial in the characterization of the first exit times. This differs 
strongly from the pure jump case which is the focus of this work. In this context, obtaining a closed 
form for the rate function is a hard task since the class of minimizers are scalar functions that 
represent shifts of the compensator of $\varepsilon \tilde N^{\frac{1}{\varepsilon}}$ on the 
nonlocal (possibly singular) component of the underlying controlled dynamics. This is an additional  
difficulty in the characterization of the first exit time in terms of the pseudo-potential. However, 
in the case of finite absolutely continuous jump measures we can solve the first exit time problem with the help of explicit formulas that we obtain for the controls. 
In other words, on an abstract level the physical intuition remains intact; 
however, since the control is given as a density w.r.t. the L\'evy measure $\nu$,
it is often hard to calculate the energy minimizing paths. 
This is the object of discussion in Section \ref{sec: extensions remarks} where we illustrate our results with several examples. The first class for which we can solve everything explicitly in terms of the coefficients of (\ref{eq:introsys}) is the finite intensity benchmark case $\nu(dz)= e^{- |z|^\beta} dz$, $\beta \gqq 2$. As a second example we introduce the natural class of Gauss-tempered $\alpha$-stable L\'{e}vy measures in the spirit of strongly tempered $\alpha$-stable L\'{e}vy measures studied in \cite{Ros07}. For this class of measures we solve the Kramers problem subject to an additional continuity property for the controlled path dynamics (cf. Hypothesis \ref{condition: generalized statement-potential} in Subsection \ref{sec: extensions remarks}). L\'{e}vy measures with compact support are another important class of measures that are covered in this setting.

\no Analogously to the Brownian case \cite{DZ98, FW98} we construct for the lower bound of the first exit time a (modified) Markov chain approximation that takes into account the topological particularities of 
the Skorokhod space on which we have the LDP. In addition, the effect of the $(1/\e)$-acceleration of the jump intensity 
enters as follows. The asymptotically exponentially negligible error estimates concerning the stickyness of the diffusion to its initial value, which in the classical Brownian case are valid for time intervals of order $1$, 
in our case only hold for time intervals of order $\e$.  

\no The article is organized as follows. We start with the exposition of the generic setting followed by
the discussion of the specific hypothesis for the LDP and the 
Kramers problem for finite intensity. It is followed by the previously mentioned Subsection \ref{sec: extensions remarks} where we extend the results to infinite intensity, discuss the additional hypotheses and present natural classes of examples including the new class of Gauss-tempered $\alpha$-stable processes. In Section 2 we establish  the 
LDP of $(X^{\varepsilon,x})_{\varepsilon>0}$ given by (\ref{eq:introsys}) on a finite time interval.
Section 3 deals with the upper and the  lower bound of the Kramers problem. The appendix essentially contains auxiliary technical results for the derivation of LDP.

\medskip

\subsection{The setting:} \label{subsec: object of study}
\paragraph{The deterministic dynamics:}

Consider the following $\cC^1$ vector field 
$b: \RR^d \ra \RR^d$, $x \in \RR^d$ and  
the deterministic dynamical system given as the solution flow $t \mapsto X^{0, x}_t$ 
of the ordinary differential equation
\begin{align} \label{eq: determinsitic ode}
\displaystyle \frac{d}{dt} X^{0,x}_t &= b(X^{0,x}_t),  \qquad  t\gqq 0, \qquad \mbox{ and }\qquad X^{0,x}_0 = x,
\end{align}
subject to the following assumptions. 

\begin{condition} \label{condition: det dynamical system} The vector field $b$ satisfies the following. 
\begin{itemize} 
\item[\textbf{A.1:}] There is a constant $c_1>0$ such that 
\begin{align} \label{eq: coercivity vector field det}
\langle b(y_1)-b(y_2), y_1-y_2 \rangle \lqq - c_1 |y_1-y_2|^2, \quad \mbox{ for all }y_1, y_2\in \RR^d.
\end{align}
\item[\textbf{A.2:}] The point $0\in\RR^d$ is critical in that $b(0)=0$. 
\end{itemize}
\end{condition}

\begin{rem}
\begin{enumerate}
 \item It is well-known that under Hypothesis \ref{condition: det dynamical system}, for 
every initial point $x \in \RR^d$ there is a unique solution $t\mapsto X^{0,x}_t$ of \eqref{eq: determinsitic ode} 
for all $t \gqq 0$. 
 \item Hypothesis \ref{condition: det dynamical system}.1
implies that $Db(x)$ is strictly negative definite for all $x\in D$. In the case of a gradient system 
$b = - \nabla \uU$ for some potential $\uU: \RR^d \ra [0, \infty)$, this is equivalent to uniform convexity.  

As a consequence, $0$ is a hyperbolic stable fixed point of the dynamical system (2) in the sense that 
there is a constant $a > 0$ such that 
all the eigenvalues $\la$ of $D^2 b(0)$ have negative real part with $\Re(\lambda) < −a < 0$. Hence due to [51]-Theorem 5.1 it follows the limit $e^{at} x(t) = 0.$
\end{enumerate}
\end{rem}

In the sequel we define  
the stochastic perturbation $\e N^{\frac{1}{\varepsilon}}$ of \eqref{eq: determinsitic ode} formally. See also \cite{BDM11} and \cite{BCD13}.  

\paragraph{The underlying noise  $\varepsilon N^{\frac{1}{\varepsilon}}$.}

Let $\mathfrak{M}$ be the space of all locally finite measures defined on the Borel 
$\sigma$-algebra $\bB(\RR^d\backslash\{0\})$. 

\no We fix a non-atomic measure $\nu \in \fM$; that is, 
$\nu(\{z\})=0$ for all $z \in \RR^d$ 
and $\nu(K)< \infty$ for every compact set $K \subset \RR^d$ with $0 \notin K$. 
Theorem I.9.1 in \cite{Ikeda Watanabe} then shows that 
the measurable space $(\mathfrak{M}, \bB(\mathfrak{M}))$ can be equipped 
with a unique non-atomic probability measure $\PP$ such that the canonical map 
$N: \mathfrak{M} \ra \mathfrak{M}$, $N(m):=m$ defines a Poisson random 
measure with intensity measure $ds \otimes \nu$ on $[0, \infty) \times \RR^d\backslash \{0\}$, 
where $ds$ denotes the Lebesgue measure on the interval $[0, \infty)$. 
We also refer the reader to Proposition 19.4 in \cite{Sato}. 
The compensated Poisson random measure of $N$ is defined by 
$\tilde N([0,s] \times A):= N([0,s] \times A) - s \nu(A)$ for all $s \gqq 0$ 
and $A \in \bB(\RR^d\backslash\{0\})$ 
such that $\nu(A)< \infty$. The expectation under $\PP$ is denoted by $\EE$. 
For all $\e>0$ we denote by $N^{\frac{1}{\e}}$ the Poisson random measure defined on the probability 
space $(\mathfrak{M}, \bB(\mathfrak{M}), \PP)$ with intensity measure $ds \otimes  \frac{1}{\e}  \nu$ 
and its compensated counterpart $\tilde N^{\frac{1}{\e}}$. 
In particular, we have $N=N^1$ and $\tilde N= \tilde N^1$. 

Consider the space $[0, \infty) \times \RR^d\backslash\{0\} \times [0, \infty)$ and 
denote by $\bar{\mathfrak{M}}$ the space of the locally finite measures 
defined on the Borel $\sigma$-algebra $\bB([0, \infty) \times\RR^d \backslash\{0\}\times [0, \infty))$. 
Analogously there is a unique probability measure $\bar \PP$ defined 
on $(\bar{\mathfrak{M}}, \bB(\bar{\mathfrak{M}}))$ 
such that the canonical map $\bar N: \bar{\mathfrak{M}} \ra \bar{\mathfrak{M}}$, 
$\bar N(\bar m):= \bar m$, is a Poisson random measure on the probability space 
$(\bar{\mathfrak{M}}, \bB(\bar{\mathfrak{M}}), \bar \PP)$ with intensity measure 
$ds \otimes \nu \otimes dr$, where $dr$ denotes the Lebesgue measure on the interval $[0, \infty)$. 
We write $\bar \EE$ for the $\bar \PP$ expectation. 

\begin{rem}
For $(t,z,r) \in [0, \infty) \times \RR^d\backslash\{0\} \times [0, \infty)$, $t$ represents 
the time variable, $z$ the spatial jump increments $z$ of the underlying L\'{e}vy process 
associated to the Poisson random measure $\bar N$
and $r$ the frequency of the jump $z$ at time $t$. 
\end{rem}

\no For any $\e>0$ the Poisson random measure $N^{\frac{1}{\e}}$ 
has the following representation as a controlled random measure with respect to $\bar N$ under $\bar \PP$. 
We have $\bar \PP$-almost surely for every $t \gqq 0$ and $A \in \bB(\RR^d\backslash\{0\})$ the identity  
\begin{align} \label{eq: PRM as controlled random measure}
N^{\frac{1}{\e}}([0,t] \times A) = \int_0^t \int_{A} \int_0^\infty \textbf{1}_{[0, \frac{1}{\e}]}(r) \bar N(ds,dz,dr).
\end{align}
For details we refer the reader to \cite{BDM11}. 

\begin{condition} \label{condition: the measure nu} 
The measure $\nu \in \mathfrak{M}$ is non-atomic and satisfies the following conditions.
\begin{itemize}
\item[\textbf{B.1:}] $\nu$ is a finite measure, $\nu(\RR^d \backslash \{0\}) < \infty$.
\item[\textbf{B.2:}] There exists $\Gamma > 0$ such that 
\begin{align} \label{eq: integrability cond measure}
\int_{B_1^c(0)} e^{\Gamma |z|^2} \nu(dz) < \infty.
\end{align}
\item[\textbf{B.3:}] The measure $\nu$ is absolutely continuous with respect to the Lebesgue measure $dz$ 
on the measurable space $(\RR^d \backslash \{0\}, \bB(\RR^d \backslash \{0\}))$ and $\frac{d \nu}{dz}(z) \neq 0$ for 
every $z \in \RR^d \backslash \{0\}$.
\end{itemize}
\end{condition}
For a discussion of Hypothesis \ref{condition: the measure nu} we refer the reader to the 
remarks after the more general Hypothesis \ref{condition:generalized statement- the measure nu} 
in Subsection \ref{sec: extensions remarks}.
\begin{rem}  From Hypothesis \ref{condition: the measure nu} it follows for any $\varepsilon>0$ 
that the jumps of the stochastic process 
$L^\varepsilon_t = \int_0^t \int_{\RR^d \backslash \{0\}} z \tilde N^{\frac{1}{\varepsilon}}(ds,dz)$ have finite intensity.
\end{rem}

\paragraph{The multiplicative coefficient.}

\no The function
$G:\RR^d \longrightarrow \RR \backslash \{0\}$ satisfies the following. 
\begin{condition} \label{condition: on the multiplicative coefficient}
 There exists $L>0$ such that for all 
  $x,y \in \RR^d$  we have 
  \begin{align*}
|G(x)- G(y)|&\lqq L |x-y| \qquad \mbox{ and } \qquad |G(x)| \lqq L(1 + |x|).
\end{align*}
\end{condition}

\paragraph{The stochastic differential equation.}  
Under Hypotheses~\ref{condition: det dynamical system}, \ref{condition: the measure nu} and \ref{condition: on the multiplicative coefficient} we consider for every $\e>0$ and $x \in \RR^d$ the following SDE 
\begin{align} \label{eq: the sde full perturbation}
X^{\e,x}_t = x + \int_0^t b(X^{\e,x}_s) ds + \varepsilon \int_0^t \int_{\RR^d \backslash \{ 0\}} G(X^{\e,x}_{s-})z  \tilde N^{\frac{1}{\e}}(ds,dz), 
\quad t \gqq 0,
\end{align}
. 

\no We denote by $(\fF_t)_{t\gqq 0}$ the filtration given  for any $t \gqq 0$ by
\begin{align*}
   \fF_t = \sigma \big( \big \{  N^1([0,s] \times A \times C  ) ~|~ 0\lqq s \lqq t, A \in \bB(\RR^d\backslash\{0\}), 
 C \in \bB([0, \infty)) \big\} \vee \nN \big),\qquad t\gqq 0, 
\end{align*}
where $\nN$ is the collection of the $\bar \PP$-null sets in $\bB([0, \infty) \times \RR^d\backslash\{ 0\} \times [0, \infty))$. 

Let $\DD([0,T], \RR^d)$ be 
the linear space of c\`{a}dl\`{a}g functions over the interval $[0, T]$, $T>0$, with values in $\RR^d$. 
It is well-known in the literature that the space $\DD([0,T], \RR^d)$ equipped with the topology generated by the $J_1$-metric $d_{J_1}$, 
known as the Skorokhod space, is a Polish space (see for instance Theorem 12.1 and Theorem 12.2 in \cite{Billingsley}).  


For the following result we cite Theorem IV-9.1  in Ikeda Watanabe and Theorem 6.4.5 of \cite{Applebaum-09}.
\begin{theorem} \label{thm: existence sol sde} 
Given $\e, T>0$, $x \in \RR^d$ and $\nu \in \fM$ 
let Hypotheses \ref{condition: det dynamical system}, \ref{condition: the measure nu} and \ref{condition: on the multiplicative coefficient} be satisfied. 
Then there is a unique c\`{a}dl\`{a}g $(\fF_t)_{t \in [0,T]}$-adapted stochastic process $(X^{\varepsilon,x}_t)_{t \in [0,T]}$ satisfying \eqref{eq: the sde full perturbation} for any $t \gqq 0$ $\bar \PP$-a.s. 
In addition, $X^{\e,x} = (X^{\e,x}_t)_{t\in [0, T]}$ is a strong Markov 
process with respect to the filtration $(\fF_t)_{t\in [0, T]}$.
In particular, due to the uniqueness of the strong solution of (6) in the sense of Definition IV-1.5 in \cite{Ikeda Watanabe} there is a $\bar \PP$-a.s. well-defined measurable map $\gG^{\varepsilon,x}: \bar{\mathcal{M}} \longrightarrow \DD([0,T];\RR^d)$ such that $\varepsilon N^{\frac{1}{\varepsilon}} \mapsto X^{\varepsilon,x}$.
\end{theorem}
  \medskip

\subsection{Statement of the main results} \label{subsection: main results}
Let the standing assumptions of Subsection \ref{subsec: object of study} be satisfied, in particular 
Hypotheses \ref{condition: det dynamical system},\ref{condition: the measure nu} and \ref{condition: on the multiplicative coefficient}.
\subsubsection{A LDP for $(X^{\varepsilon,x})_{\varepsilon>0}$}
Whenever possible without confusion we shall drop the index for the initial condition $x$. 
In this paragraph we fix some notation and introduce the necessary objects for the statement of the LDP
of $(X^{ \e, x})_{\e >0}$ following \cite{BDM11} and \cite{BCD13}. 
For fixed $T>0$ and a measurable function $g: [0,T] \times \RR^d \backslash \{0\} \ra [0, \infty)$ we define the entropy functional  by 
\begin{align} \label{eq: entropy functional}
\eE_T(g) := \int_0^T \int_{\RR^d \backslash \{ 0\}} (g(s,z) \ln g(s,z) - g(s,z)+1) \nu(dz) ds.
\end{align}
For every $M \gqq 0$ we define the sublevel sets of the functional $\eE_T$ by 
\begin{align} \label{eq: the sublevel sets}
\mathfrak{S}^M := \Big \{ g: [0,T] \times \RR^d \backslash \{0\} \lra [0, \infty) \mbox{ mb}~\big|~ \eE_T(g) \lqq M \Big \} \quad \mbox{ and set } \quad \mathfrak{S} := \bigcup_{M \gqq 0} \mathfrak{S}^M. 
\end{align}
Given $T>0$, $x \in \RR^d$ and $g \in \mathfrak{S}$ we 
consider the controlled integral equation
\begin{align} \label{eq: controlled ODE}
U^{g} (t;x) = x + \int_0^t b(U^g(s;x))ds + \int_0^t \int_{\RR^d \backslash \{ 0\}} G(U^g (s;x))(g(s,z)-1) z \nu(dz)ds, ~t \in [0,T]. 
\end{align}
It is standard in the literature (see Theorem 3.7 in \cite{BCD13})  
that the equation \eqref{eq: controlled ODE} has 
a unique solution $U^{g} \in C([0,T], \RR^d)$ and it satisfies the uniform bound 
\begin{align} \label{eq: uniform bound controlled odes}
\displaystyle \sup_{t \in [0,T]} \displaystyle \sup_{g \in \mathfrak{S}^M} |U^g (t;x)| < \infty \qquad \mbox{ for all }M>0. 
\end{align}
In particular, the map $\gG^{0, x}: \mathfrak{S} \ra C([0,T], \RR^d)$, $g \mapsto \gG^{0,x}(g):= U^{g}(\cdot \,;x)$ is well-defined  for any fixed $x\in \RR^d$.
 For $\varphi \in C([0,T], \RR^d)$ we define
$\mathfrak{S}_{\varphi, x} := \{ g \in \mathfrak{S} ~|~ \varphi= \gG^{0,x}(g)\}$  the preimage of $\varphi$ under $\gG^{0,x}$ and set   $\JJ_{x, T}: \DD([0,T], \RR^d) \ra [0, \infty]$ 
\begin{align} \label{eq: rate function}
\JJ_{x,T}(\varphi) := \displaystyle \inf_{g \in \mathfrak{S}_{\varphi, x}} \eE_T(g),
\end{align}
with the convention that $\inf \emptyset=\infty$. 

\begin{theorem} \label{thm: LDP full process} Let Hypotheses \ref{condition: det dynamical system}, \ref{condition: the measure nu} and \ref{condition: on the multiplicative coefficient}
be satisfied for some $\nu\in \fM$, $T>0$ and $x \in D$ fixed and let 
$X^{\e,x} = (X^{\e, x}_t)_{t\in [0, T]}, \varepsilon>0$, 
be the  strong solution of (\ref{eq: the sde full perturbation}) given in Theorem \ref{thm: existence sol sde}. 
Then the family $(X^{\e, x})_{\e>0}$ satisfies a LDP with the good rate function $\JJ_{x,T}$ given by \eqref{eq: rate function} 
in the Skorokhod space $\DD([0,T], \RR^d)$. This means that, for any $a \gqq 0$ the 
sublevel set $\{ \JJ_{x,T} \lqq a\}$ is compact in $\DD([0,T],\RR^d)$ and for any 
$G \subset \DD([0,T], \RR^d)$ open and $F \subset \DD([0,T], \RR^d)$ closed,
\begin{align*}
\displaystyle \liminf_{\varepsilon \ra 0} \varepsilon \ln \bar \PP(X^{\varepsilon,x} \in G) &\gqq - \displaystyle \inf_{\varphi \in G} \JJ_{x,T}(\varphi) \text{ and } \\
\displaystyle \limsup_{\varepsilon \ra 0} \varepsilon \ln \bar \PP(X^{\varepsilon,x} \in F) & \lqq - \displaystyle \inf_{\varphi \in F} \JJ_{x,T}(\varphi).
\end{align*}
\end{theorem}

\medskip
\subsubsection{The asymptotic first exit problem of $X^{\varepsilon,x}$ from $D$ as $\varepsilon \ra 0$} \label{subsection: asymptotic exit time}

 We make the additional assumptions as follows.
\begin{condition}\label{condition: domain}  Let us consider a bounded domain $D \subset \RR^d$ 
  with $0 \in D$, $\partial D \in \cC^1$ and that $b$ is inward-pointing on $\partial D$, that is, 
\begin{align*}
\langle b(z), n(z) \rangle < 0, \quad \mbox{ for all } z \in \partial D, 
\end{align*}
where the vector field $\pd D \ni z \mapsto n(z)\in \RR^d$ denotes the outer normal on $\pd D$.
\end{condition}
\begin{rem}
The first statement of Hypothesis \ref{condition: domain} implies that 
the solution of \eqref{eq: determinsitic ode} is 
positive invariant on $\bar D$, that is, for all $x \in \bar D$, we have 
$X^{0, x}_t \in D$ for all $t \gqq 0$ 
and $X^{0, x}_t \rightarrow 0$ as $ t \ra \infty$.  
\end{rem}
\no Given $\e>0$, $x \in D$ and $\nu \in \fM$ satisfying 
Hypotheses \ref{condition: det dynamical system}, \ref{condition: the measure nu}, 
\ref{condition: on the multiplicative coefficient} and \ref{condition: domain} 
we define the first exit time of the solution $X^{\e,x}$ of \eqref{eq: the sde full perturbation} 
from $D$ 
\begin{align} \label{eq: the first exit time}
\sigma^\e(x) := \inf \{ t \gqq 0 ~|~ X^{\e,x}_t \notin D \}  
\end{align}
and the first exit location $X^{\e, x}_{\sigma^\e(x)}$. 

\no The function $V$ quantifying the cost of shifting the intensity jump measure by a scalar control $g$ and steering $U^g(t;x)$ from its initial position $x$ to some $z\in \RR^d$ in cheapest time is defined as 
 \begin{align} \label{eq: running cost}
 V(x,z) := \displaystyle \inf_{T>0} \inf \Big \{ \JJ_{x,T} (\varphi) ~|~ \varphi \in \DD([0,T], \RR^d): 
 \varphi(T)=z \Big \} \quad \text{ for } x,z \in \RR^d.
 \end{align}
 The function $V(0,z)$ is called the quasi-potential of the stable state $0$ with potential height 
 \begin{align} \label{eq: the potential}
 \bar{V}:= \displaystyle \inf_{z \notin D} V(0,z).
 \end{align}
We are ready to present our main result. The proof is the combination of Corollary \ref{corollary: finite height}, 
Theorem \ref{chpt2: thm upper bound first exit time}, Theorem \ref{chpt2: thm lower bound first exit time} 
and Remark \ref{rem: exit location} in Section \ref{sec: first exit}.
\begin{theorem} \label{thm: first exit time}
Let Hypotheses  \ref{condition: det dynamical system}, \ref{condition: the measure nu}, 
\ref{condition: on the multiplicative coefficient} and \ref{condition: domain} be satisfied. 
Then $\bar V < \infty$ and we obtain the following result.
\begin{enumerate}
 \item \label{thm: exit time}
For any $x \in D$ and $\delta>0$, we have 
\begin{align} \label{eq: first exit time-limit1}
\displaystyle \lim_{\e \ra 0} \bar \PP \Big ( e^{\frac{\bar V - \delta}{\e}}< \sigma^\e(x) < e^{\frac{\bar V + \delta}{\e}} \Big )=1.
\end{align}
Furthermore, for all $x \in D$ it follows $\displaystyle \lim_{\e \ra 0} \e \ln \bar \EE[\sigma^\e(x)]= \bar V.$

\item \label{thm: location exit}
For any closed set $F \subset D^c$ satisfying $\displaystyle \inf_{z \in F} V(0,z) > \bar V$ 
and any $x \in D$, we have 
 \begin{align} \label{eq: first exit time-limit2}
 \displaystyle \lim_{\e \ra 0} \bar \PP \Big ( X^{\e,x}_{\sigma^\e(x)} \in F \Big )=0.
 \end{align}
In particular, if $\bar V$ is taken by a unique point $z^{*} \in D^c$, 
it follows, for any $x \in D$ and $\delta>0$, that 
\begin{align} \label{eq: first exit time-limit3}
\displaystyle \lim_{\e \ra 0} \bar \PP( |X^{\e,x}_{\sigma^\e(x)} - z^{*}| < \delta )=1.
\end{align}
\end{enumerate}
 \end{theorem}

\bigskip

\subsection{Extensions and remarks} \label{sec: extensions remarks}
In this subsection we discuss the analogous statements of Theorem \ref{thm: LDP full process} and 
Theorem \ref{thm: first exit time} in a  general framework with  a vector-valued function 
$G:\RR^d \times \RR^d \longrightarrow \RR^d$ and a L\'{e}vy measure $\nu$ with infinite intensity. 
It is the aim of this subsection to present a sufficient condition for the LDP and the solution of the Kramers 
problem of the following family of processes $(X^{\varepsilon,x})_{\varepsilon >0}$. Such a condition is given below as
Hypothesis \ref{condition: generalized statement-potential} and  is formulated as a continuity property for the controlled dynamics. This turns out to be hard to be verified in general and needs to be studied case by case.

\no For every $\varepsilon>0$ and $x \in D$ we consider the unique strong solution $(X^{\varepsilon,x}_t)_{t \gqq 0}$ of
\begin{align}\label{eq: generalized statement- SDE}
 X^{\e,x}_t = x + \int_0^t b(X^{\e,x}_s) ds + \int_0^t \int_{\RR^d\backslash\{0\}} G(X^{\e,x}_{s-},z) \;\e \tilde N^{\frac{1}{\e}}(ds,dz), 
\quad t \gqq 0.
\end{align}
\no The function $b$ remains unchanged satisfying Hypothesis \ref{condition: det dynamical system}
 and \ref{condition: domain}.  For every $\varepsilon>0$ $\varepsilon \tilde N^{\frac{1}{\varepsilon}}$ 
 is a compensated Poisson random measure defined on $(\mathfrak{M}, \bB(\mathfrak{M}), \bar \PP)$ with 
 compensator given by $ds \otimes \frac{1}{\varepsilon} \nu(dz)$ where $\nu \in \mathfrak{M}$ satisfies 
 the following assumption which replaces Hypothesis \ref{condition: the measure nu}. 
 \begin{condition} \label{condition:generalized statement- the measure nu} 
The measure $\nu \in \mathfrak{M}$ is non-atomic and satisfies the following conditions.
\begin{itemize}
\item[\textbf{E.1:}] The measure $\nu$ is a L\'{e}vy measure, i.e. $\int_{\RR^d \backslash \{0\}} (1 \wedge |z|^2) \nu(dz)< \infty$.
\item[\textbf{E.2:}] $\nu \in \mathfrak{M}$ satisfies  $\int_{B^c_1(0)} e^{\Gamma |z|^2} \nu(dz) < \infty$ for some $\Gamma>0$.
\item[\textbf{E.3:}] There are positive constants $p, q$ such that 
\begin{equation}\label{E.3}
\limsup_{\e\ra  0+} ~|\ln(\e)|^p \cdot\int_{|z|\lqq \sqrt{\e} |\ln(\e)|^q} |z|^2 \nu(dz) < \infty. 
\end{equation}
\end{itemize}
\end{condition}

\begin{rem} 
\begin{enumerate}
\item We stress that Hypothesis \ref{condition:generalized statement- the measure nu} is weaker 
than Hypothesis \ref{condition: the measure nu} and covers a wide class of L\'{e}vy measures. 
The measure $\nu(dz)= e^{- |z|^\beta}dz$, $\beta \gqq 2$, is an important benchmark case and is 
the model case covered by Hypothesis \ref{condition: the measure nu}. 
In the literature it is known as a super-exponentially light jump measure.
\item More generally Hypothesis \ref{condition:generalized statement- the measure nu} covers a class of L\'{e}vy measures which mimics the strongly tempered 
$\alpha$-stable measures introduced by Rosi\'nski \cite{Ros07}, however, with a Gaussian damping term in order to guarantee the moment condition in Hypothesis B.2 (and Hypothesis E.2 resp.). We define a \textit{Gauss-tempered $\alpha$-stable L\'{e}vy measure} $\nu$ as follows. For the radial coordinate $r=|z|$, for some 
measure $R \in \fM$, $\gamma>0$ and $\alpha \in (0, 2)$ satisfying 
$\int_{\RR^d \backslash \{0\}} |z|^{\alpha} R(dz) < \infty$ we set
\begin{equation}\label{d:Gaussian-tempered}
\nu(A) := \int_{\RR^d \backslash \{0\}} \int_0^\infty \ind_A(rz) 
\frac{e^{- \frac{\gamma}{2} r^2}}{r^{\al+1}}  dr  R(dz), \quad A \in \bB(\RR^d \backslash \{0\}).
\end{equation} 
\item Hypothesis \ref{condition:generalized statement- the measure nu} is satisfied by 
L\'{e}vy measures with compact support as in the seminal paper  \cite{BDM11}.
\item Hypothesis \ref{condition:generalized statement- the measure nu}.3 is a non-degeneracy condition on the pole of 
the L\'evy measure in $0$. It is trivially satisfied for finite L\'evy measures (Hypothesis \ref{condition: the measure nu}.1). In addition, it is satisfied for any Gauss-tempered $\alpha$-stable L\'{e}vy measure $\nu$ 
and any $\al\in (0,2)$ as defined in \eqref{d:Gaussian-tempered}
\begin{align*}
\int_{|z|\lqq \sqrt{\e} |\ln(\e)|^q} |z|^2 \nu(dz) 
&= \int_{\RR^d \backslash \{0\}} \int_0^\infty |rv|^2 \ind\{|rv|\lqq \sqrt{\e} |\ln(\e)|^q\} \frac{e^{- \frac{\gamma}{2} r^2}}{r^{\al+1}}  dr  R(dv)\\
&\lqq \int_{\RR^d \backslash \{0\}} |v|^2 \int_0^\infty |r|^{1-\al} \ind\{|r|\lqq \frac{\sqrt{\e} |\ln(\e)|^q}{|v|}\}   dr  R(dv)\\
&= \int_{\RR^d \backslash \{0\}} |v|^2 \frac{1}{2-\al} \Big(\frac{\sqrt{\e} |\ln(\e)|^q}{|v|}\Big)^{2-\al}   R(dv)\\
&= \frac{\big(\sqrt{\e} |\ln(\e)|^q\big)^{2-\al}}{2-\al} \int_{\RR^d \backslash \{0\}} |v|^{\al}R(dv).
\end{align*}
The expression on the right-hand side is finite by definition 
and tends to $0$ as $\e\ra 0$ with polynomial rate and consequently satisfies the relation \eqref{E.3}.
\end{enumerate}
\end{rem}
\no The vector-valued multiplicative coefficient 
$G:\RR^d \times \RR^d \ra \RR^d$ satisfies the following assumption that replaces Hypothesis \ref{condition: on the multiplicative coefficient}.
\begin{condition} \label{condition: generalized statement-on the multiplicative coefficient}
 There exists $L>0$ such that for all 
  $x,y,z \in \RR^d$  we have 
  \begin{align*}
|G(x,z)- G(y,z)|&\lqq L |z| |x-y| \qquad \mbox{ and } \qquad |G(x,z)| \lqq L |z| (1 + |x|).
\end{align*}
\end{condition}

\begin{rem} Under Hypotheses \ref{condition: det dynamical system}, \ref{condition:generalized statement- the measure nu} (which covers \ref{condition: the measure nu}) and \ref{condition: generalized statement-on the multiplicative coefficient} (which covers \ref{condition: on the multiplicative coefficient}) the statements of Theorem \ref{thm: existence sol sde} remain valid.
\end{rem}

\no The following hypothesis is the fundamental assumption of this subsection.
\begin{condition} \label{condition: generalized statement-potential}  
For every $\rho_0>0$ there exist a constant 
$M>0$ and a non-decreasing function $\xi:[0, \rho_0] \ra \RR^{+}$ 
with $\lim_{\rho \ra 0} \xi(\rho)=0$ satisfying the following. 
For all $x_0, y_0 \in \RR^d$ such that $|x_0- y_0| \lqq \rho_0$ there exist 
$\Phi \in C([0, \xi(\rho_0)],\RR^d)$ and $g \in \mathfrak{S}^M$ 
such that $\Phi(\xi(\rho_0))=y_0$ and solving
\begin{align} \label{eq: generalized statement-cond on potential: controllability cond C1}
 \Phi(t) = x_0 + \int_0^t b(\Phi(s)) ds + \int_0^t \int_{\RR^d \backslash \{ 0\}} G(\Phi(s),z) (g(s,z)-1) \nu(dz) ds, \quad t \in [0, \xi(\rho_0)].
\end{align}
\end{condition}
\begin{rem} \label{rem: suffering MC}
Hypothesis \ref{condition: generalized statement-potential} is covered under the assumptions stated in Theorem \ref{thm: first exit time}. This is proved in Proposition \ref{prop: continuity pps cost functional}. 
\end{rem}
\paragraph*{Construction of the good rate function and the potential.} 
Let Hypotheses \ref{condition: det dynamical system}, \ref{condition: domain}, 
\ref{condition:generalized statement- the measure nu}, \ref{condition: generalized statement-on the multiplicative coefficient} 
and \ref{condition: generalized statement-potential} be satisfied.  For fixed $x \in \RR^d$, $T>0$, 
define $\eE_T$ by (\ref{eq: entropy functional}) and $\mathfrak{S}$ by (\ref{eq: the sublevel sets}). 
Let $\gG^{0,x}: \mathfrak{S} \longrightarrow C([0,T];\RR^d)$ defined as $\gG^{0,x}(g):= U^{g}(\cdot\,; x)$ 
where the function $U^g$ is the unique solution of the  following controlled integral equation
\begin{align} \label{eq: generalized statement-controlled ODE}
U^{g} (t;x) = x + \int_0^t b(U^g(s;x))ds + \int_0^t \int_{\RR^d \backslash \{ 0\}} G(U^g (s;x),z)(g(s,z)-1) \nu(dz)ds, ~t \in [0,T]. 
\end{align}
With this notation the definition of the good rate function $\JJ_{x,T}$ in (\ref{eq: rate function}) 
and the potential height $\bar V$  in (\ref{eq: the potential}) remain unchanged and the results of 
Theorem \ref{thm: LDP full process} and Theorem \ref{thm: first exit time} are carried over in the next theorem.

\begin{theorem}
For fixed $T>0$,  $x \in \RR^d$ and $\nu \in \mathfrak{M}$ let Hypotheses \ref{condition: det dynamical system},
\ref{condition:generalized statement- the measure nu}, \ref{condition: generalized statement-on the multiplicative coefficient} 
and \ref{condition: generalized statement-potential} be satisfied. Let $(X^{\varepsilon,x})_{\varepsilon>0}$ 
be the family of strong solutions of (\ref{eq: generalized statement- SDE}). Then we have the following.
\begin{enumerate}  \item The family $(X^{\varepsilon,x})_{\varepsilon>0}$ satisfies a LDP in $\DD([0,T], \RR^d)$ 
with the good rate function $\JJ_{x,T}$ given by~(\ref{eq: rate function}).
\item For any domain $D \subset \RR^d$ satisfying Hypothesis \ref{condition: domain} and $x \in D$ the results (\ref{eq: first exit time-limit1})-(\ref{eq: first exit time-limit3}) of Theorem \ref{thm: first exit time} hold for the respective exit time $\sigma^\varepsilon(x)$.
\end{enumerate}
\end{theorem}
\no In this setting the proof of the preceding results  is virtually identical to the proofs of 
Theorem \ref{thm: LDP full process} and Theorem \ref{thm: first exit time} and is therefore omitted.

\begin{rem} \label{rem: comment_measure}
We comment on the assumptions stated on Hypothesis B and further generalizations made on Hypothesis E for the L\'{e}vy measure $\nu$.
\begin{enumerate}
\item The exponential integrability assumption of Hypothesis \ref{condition: the measure nu}.2 (and analogously \ref{condition:generalized statement- the measure nu}.2) is sufficient to verify the abstract condition for a LDP in \cite{BDM11} as explained in the introduction. More precisely, it is used in the proof of auxiliary a-priori bounds for the controlled integrals stated in Lemma~\ref{lemma: integrability controls}. It is proved 
under the general Hypothesis \ref{condition:generalized statement- the measure nu}. 
In particular, the exponential Young inequality applied there seems near optimal and hard to relax. On the other hand even to derive a LDP for one dimensional empirical means of i.i.d. 
random variables some exponential integrability of the underlying law is necessary. 
See Lemma 2.2.5 (b) in the proof of the Cram\'er Theorem~2.3.3 in \cite{DZ98}.

\item A priori, the extension from Hypothesis B to Hypothesis E bears the difficulty, 
that it is hard to solve the control problem \eqref{eq: cond on potential: controllability cond C1} which under Hypotheses \ref{condition: the measure nu}.1 and \ref{condition: the measure nu}.3 can be solved explicitly in Proposition \ref{condition: potential}. 
As a consequence we state its solvability 
in Hypothesis \ref{condition: generalized statement-potential} as a proper hypothesis, 
since its solution may require rather different 
techniques to be verified in particular situations.
\end{enumerate} 
\end{rem}

\bigskip
\section{The large deviations principle} \label{sec: the LDP}
In this section we prove Theorem \ref{thm: LDP full process}. Let Hypotheses \ref{condition: det dynamical system}, 
\ref{condition: the measure nu} and \ref{condition: on the multiplicative coefficient} be satisfied for some $\nu\in \mathfrak{M}$.  
For every $\e>0$, $T>0$ and $x \in \RR^d$, we consider the strong solution $(X^{\e,x})_{t \in [0,T]}$ 
of the SDE \eqref{eq: the sde full perturbation}. By Theorem \ref{thm: existence sol sde}  the map 
$$
\gG^{\e,x} : \mathfrak{M} \ra \DD([0,T], \RR^d), \quad 
\gG^{\e,x} (\e N^{\frac{1}{\e}}):= X^{\e,x},
$$ 
is measurable with respect to the Borel sigma algebras associated to 
the vague convergence topology in $\mathfrak{M}$  
and the (Skorokhod) $J_1$ topology in $\DD([0,T], \RR^d)$.  
On the other hand, given $g \in \mathfrak{S}$, 
the wellposedness of the integral equation \eqref{eq: controlled ODE} 
yields the existence of a  measurable map 
$$
\gG^{0, x} :  \mathfrak{S} \ra C([0,T], \mathbb{R}^d), \quad \gG^{0,x} (g):= U^{g}(\cdot\,,x).
$$  
The main task in the proof of Theorem \ref{thm: LDP full process} is the 
verification of the two statements of Condition 2.2 in \cite{BCD13} for  $(\gG^{\varepsilon,x})_{\varepsilon>0}$ and $\gG^{0,x}$,
which combined imply the LDP (cf. Theorem 2.4 in \cite{BCD13}). We essentially 
follow the notation introduced in \cite{BDM11} and  \cite{BCD13}. 
\subsection{The weak convergence approach}
\paragraph*{Notation.}
Denote by $\bar \pP$ the predictable $\sigma$-field 
on $[0, T] \times \bar{\mathfrak{M}}$ with respect to the filtration $(\fF_t)_{t\in [0, T]}$. 
We define the space of positive (random) controls in $\bar \fM$
\begin{align*}
\bar{\mathfrak{A}}^+ := \Big \{ \varphi: [0,T] \times \RR^d \backslash \{0\} \times \bar{\mathfrak{M}} 
\ra [0, \infty) ~|~ \varphi \text{ is } (\bar \pP \otimes \bB(\RR^d \backslash \{0\}), \bB([0, \infty)))-\text{measurable } \Big \}.
\end{align*}
Given a covering of $\RR^d \backslash \{0\}$ by compact sets  $(K_n)_{n \in \NN}$ 
we define the set of the $n$-cutoff positive (random) controls
\begin{align*}
\bar{\mathfrak{A}}^+_{b,n} &:= \Big\{  \varphi \in \bar \fA^+ ~|~ \varphi(t, x, \bar m) 
\begin{cases}\in [\frac{1}{n}, n], & x\in K_n, \\ =1, & x\in K_n^c, \end{cases} 
\text{ for all } (t, \bar m) \in [0,T] \times \bar{\mathfrak{M}}\Big\}. 
 \end{align*} 
The set of positive bounded controls is then given by 
$$\bar{\mathfrak{A}}^+_b := \displaystyle \bigcup_{n \in \NN} \bar{\mathfrak{A}}^+_{b,n} 
\qquad \mbox{ and }\qquad \mathfrak{U}_{+}^{M} := \{\varphi \in \bar{\mathfrak{A}}_b^+: 
\varphi(., .,\bar{m}) \in \mathfrak{S}^M \quad \bar \PP-\text{a.s }  \}, \quad M>0,$$ 
is the set of positive bounded random controls whose entropy functional is $\bar \PP$-a.s. bounded by $M$. 
We associate to every $g \in \mathfrak{S}^M$ the measure 
\[
\bB([0,T] \times \RR^d \backslash \{0\}) \ni A \mapsto \nu_T^g(A):= \int_A g(s,z) \nu(dz)ds
\] 
and identify $\mathfrak{S}^M$ with the space of associated measures $\{ \nu_T^g ~|~g \in \mathfrak{S}^M \} \subset \mathfrak{M}$ 
equipped with the topology induced by the vague convergence on  $\mathfrak{M}$. 
We refer the reader to  Lemma 5.1 in \cite{BCD13} 
which ensures that this identity produces 
a topology in $\mathfrak{S}^M$ under which $\mathfrak{S}^M$ turns out to be compact. 

\no For any fixed $M>0$ and a family $(\varphi_\e)_{\e >0} \subset \mathfrak{U}_+^M$  
 we set $\psi_{\e}:= \frac{1}{\varphi_{\e}}$. 
The random measure $N^{\frac{\varphi_\e}{\e}}$ is a controlled random measure given by 
\begin{align*}
N^{\frac{\varphi_\e}{\e} } ([0,t] \times U) := \int_0^t \int_U \int_0^{\infty} 
\mathbf{1}_{[0, \frac{\varphi_\e}{\e}(s,z)]}(r) \bar N (ds,dz,dr) \quad \text{for all } t \in [0,T], U \in \bB(\RR^d \backslash \{0\}).
\end{align*}  
Recall that the canonical map $\bar N: \mathfrak{M} \ra \mathfrak{M}$, $\bar N(\bar m):= \bar m$ 
is the Poisson random measure defined on $(\bar{\mathfrak{M}},\bB( \bar{\mathfrak{M}}), \bar \PP)$ with 
intensity measure $ds \otimes \nu \otimes dr$. 

\no Since $\varphi_\e \in \mathfrak{U}^{M}_+$ yields that $\varphi_\e$ is bounded from below 
and above on a compact set in $[0,T] \times \RR^d \backslash \{0\}$ and $\varphi_\e=1$ outside 
of that compact,  we can use Girsanov's theorem in the form of  Lemma 2.3 in \cite{BDM11}.
Therefore the Doleans-Dade exponential of $\psi_\e$ with respect to $\bar N$ 
under $\bar \PP$ defined for any $t \in [0,T]$ by
  \begin{align*}
\mathfrak{E}(\psi_\e)(t) &:= \exp \Big (  \int_0^t \int_{\RR^d \backslash \{0\}} \int_0^{\frac{1}{\e}} \ln \psi_{\e} (s,z) \bar{N} (ds,dz,dr) 
  + \int_0^t \int_{\RR^d \backslash \{0\}} \int_0^{\frac{1}{\e}} ( - \psi_{\e} (s,z) + 1  )dr \nu(dz)ds \Big )
  \end{align*} is an $(\fF_t)_{t\in [0, T]}$ - martingale under $\bar \PP$.
In addition, the measure 
  \begin{align*}
 \mathcal{B}(\bar{\mathfrak{M}}) \ni G \mapsto \QQ ^\e_T(G) := \int_G \mathfrak{E} (\psi_\e) (T)  d \bar{\PP} 
 \end{align*}
is a probability measure on $(\bar{\mathfrak{M}}, \mathcal{B}(\bar{\mathfrak{M}}))$. Furthermore, the measures $\bar{\PP}$ and $\QQ^\e_T$ 
are mutually absolutely continuous and the controlled random measure $\e N^{\frac{1}{\e} \varphi_\e}$ under $\QQ^\e_T$ 
has the same law as $\e N^{\frac{1}{\e}}$ under $\bar{\PP}$ on $(\bar{\mathfrak{M}}, \bB(\bar{\mathfrak{M}}))$. 
For more details we refer the reader to Lemma 2.3 in \cite{BDM11} and further references given there. Denote by  $\tilde X^{\e,x}:= \mathcal{G}^{\e,x} (\e N^{\frac{\varphi_{\e}}{\e}} )$ the unique strong solution 
of the following controlled SDE
\begin{align} \label{chpt2: eq controlled sde}
\tilde X^{\e,x}_t = x + \int_0^t b(\tilde X_s^{\e,x})ds 
+ \int_0^t \int_{\RR^d \backslash \{0\}} G( \tilde X_{s-}^{\e,x})z \big (\e N^{\frac{\varphi_{\e}}{\e}} (ds,dz) - \nu(dz) ds \big ), \quad t \in [0,T].
\end{align}
\bigskip

\paragraph*{Technical estimates:} \label{subsection: preliminaries}

The following lemma is crucial in the proof of Theorem \ref{thm: LDP full process} and the reader 
can find its proof in the Subsection  \ref{sec: preliminaries LDP} of the appendix.
\begin{lemma}  \label{lemma: integrability controls} 
  Let $\nu \in \mathfrak{M}$ satisfy Hypothesis \ref{condition:generalized statement- the measure nu}.
  Then for any $M, T>0$ and $x \in \RR^d$ we have the following statements 
\begin{align} 
&\sup_{g \in \mathfrak{S}^M} \iint_{[0,T] \times \RR^d \backslash \{0\}} |z|^2 g(s,z) \nu(dz) ds <  \infty, 
\quad  \label{eq: integrability controls- first order} \\[3mm]
& \sup_{g \in \mathfrak{S}^M} \iint_{[0,T] \times \RR^d \backslash \{0\}} |z|| g(s,z)-1| \nu(dz) ds <  
\infty \quad \text{and}  \label{eq: integrability controls- first order1}   \\[3mm]
&\lim_{\delta \rightarrow 0} \sup_{g \in \mathfrak{S}^M}  \sup_{\substack{0 \lqq  t < t' \lqq T \\ |t-t'| \lqq \delta}} 
~\iint_{[t,t'] \times \RR^d \backslash \{0\}} |z| |g(s,z)-1| \nu(dz) ds= 0.\label{eq: integrability controls- limit zero}
\end{align}
\end{lemma}
\noindent The process $\ti X^{\e, x}$ introduced in (\ref{chpt2: eq controlled sde}) has the following localization property used in the sequel.
 \begin{proposition} \label{proposition: priori estimate controlled process} 
 Let the hypotheses of Theorem \ref{thm: LDP full process} be satisfied. 
 Then for any $M>0$, any family $(\varphi_{\e})_{\e >0}$, $\varphi_\e \in \mathfrak{U}^{M}_+$, any function  
 $\rR: (0, 1] \ra (0, \infty)$ satisfying the limits $\lim_{\e\ra 0+}\rR(\e) = \infty$ and $\lim_{\e \ra 0+} \e \rR^2(\e) =0$, 
 $x\in \RR^d$ and $T>0$ we have the following. There exist constants $\e_0 \in (0, 1]$ and $C>0$ such that $\e \in (0, \e_0]$ implies 
\begin{align} \label{chpt2: eq a priori estimate controlled  process}
\bar \PP\big( \displaystyle \sup_{s\in [0, T]} |\tilde X_s^{\e,x}| >  \rR (\e) \big ) \lqq 2 e^{- \frac{1}{2} \rR(\e)} + C \e \rR(\e).
 \end{align} 
\end{proposition}
\no The proof uses an exponential Bernstein-type inequality for martingales from \cite{DZ01} 
and given in Subsection  \ref{sec: preliminaries LDP} of the appendix. 

\no For any $\varepsilon>0$ and $\rR:(0,1] \longrightarrow (0,\infty)$ satisfying the limits $\rR(\varepsilon) \ra \infty$ 
and $\varepsilon \rR^2(\varepsilon) \ra 0$ as $\varepsilon \ra 0$, such as in the statement of 
Proposition \ref{proposition: priori estimate controlled process}, let us define the $(\fF_t)_{t \in [0,T]}$-stopping time
\begin{align} \label{eq: the stopping time}
\tilde \tau^\varepsilon_{\rR(\varepsilon)} := \inf \{ t \gqq 0  ~|~ \tilde X^{\varepsilon,x}_t \notin B_{\rR(\varepsilon)}(0)\} \wedge T.
\end{align}
\begin{proposition} \label{proposition: a priori bound}
For any $\varepsilon>0$, $x \in \RR^d$ and $\rR:(0,1] \longrightarrow (0,\infty)$ given as 
in Proposition \ref{proposition: priori estimate controlled process} the following holds. 
There exists $\Lambda:=\Lambda(L,M,T,x)>0$ and $\varepsilon_1>0$ such that
\begin{align} \label{eq: a priori bound controlled processes}
\displaystyle \sup_{\varepsilon \in (0, \varepsilon_1]} \bar \EE \Big [ \displaystyle \sup_{ t \in [0, \tilde \tau^\varepsilon_{\rR(\varepsilon)}]} |\tilde X^{\varepsilon,x}_t|^2 \Big ] \lqq \Lambda.
\end{align}
\end{proposition}
\no The preceding estimate is proved in Subsection  \ref{sec: preliminaries LDP} of the appendix.

\bigskip

\subsection{Proof of Theorem \ref{thm: LDP full process}}

The following proposition is a continuity statement of the map $\gG^{0,x}$ for any $x \in \RR^d$.
\begin{proposition} \label{proposition: first condition LDP big jumps}
For every $x \in \RR^d$, $M \gqq 0$ and $n \in \NN$ let $g_n, g \in \mathfrak{S}^M$ 
such that $\nu_T^{g_n} \ra \nu_T^g$ in the vague topology of $\mathfrak{S}$ as $n \ra \infty$. Then
there exists a subsequence $(g_{n_k})_{k \in \NN} \subset (g_n)_{n \in \NN}$ such that 
$$\mathcal{G}^{0,x} (g_{n_k}) \rightarrow \mathcal{G}^{0,x} (g), \quad \text{as } 
k \ra \infty $$ in the uniform topology of $C([0,T], \RR^d)$.
 \end{proposition}
 \begin{proof}
\no For convenience we drop the dependence on the parameter $x \in \RR^d$ in what follows.
We set $U_n:= U^{g_n} = \mathcal{G}^0 (g_n)$.  
Estimate \eqref{eq: uniform bound controlled odes} 
yields a constant $K >0$ such that 
\begin{align} \label{chpt2: eq ldp thm bound deterministic controlled process}
\displaystyle \sup_{n \in \NN} \displaystyle \sup_{t \in [0, T]} |U_n(t;x)| \lqq K . 
\end{align}
Due to  \eqref{eq: integrability controls- limit zero} it follows
  \begin{align*}
 \displaystyle \lim_{\delta \rightarrow 0} \displaystyle  \sup_{n \in \NN} \displaystyle \sup_{|t-s| \lqq \delta} |U_n(t;x) - U_n(s;x)| = 0,
  \end{align*}
 which implies that $(U_n)_{n \in \NN}$ is a family of equicontinuous 
  uniformly bounded functions in $C([0,T], \mathbb{R}^d)$. 
The Arzel\`{a}-Ascoli compactness theorem yields a limit 
in the uniform topology $U \in C([0,T], \RR^d)$ 
for some subsequence $(g_{n_k})_{k \in \NN} \subset (g_n)_{n \in \NN}$. By the uniform 
estimate \eqref{chpt2: eq ldp thm bound deterministic controlled process}, the continuity of the functions $b$ and $G$ 
and \eqref{eq: integrability controls- first order} dominated convergence yields 
\begin{align}\label{chpt2: eq ldp  limit eq}
U(t;x) = x + \int_0^t b(U(s;x)) \mathrm{d}s 
+ \int_0^t \int_{\RR^d \backslash \{0\}} G(U(s;x)) (g(s,z)-1) z \nu(\mathrm{d}z) \mathrm{d}s \quad \text{for all } t \in [0,T].
\end{align}
The uniqueness of solution of \eqref{eq: controlled ODE} implies that $U = U^g = \gG^0(g)$.
\end{proof}
\no For $\varepsilon>0$, $x \in \RR^d$, $M \gqq 0$ and $\varphi \in \mathfrak{U}_{+}^{M}$ the 
following result is a weak law of large numbers type of statement for the measurable maps $\gG^{\varepsilon,x}$ 
under the action of the controlled random measures $\varepsilon \tilde N^{\frac{\varphi}{\varepsilon}}$.
 \begin{proposition} \label{proposition: second part ldp big jumps process}
 Given $M  \gqq 0$ let $\varphi \in \mathfrak{U}_+^M$ and $(\varphi_{\e})_{\varepsilon>0} \subset \mathfrak{U}_{+}^M$ 
 such that $\varphi_{\e} \Rightarrow \varphi$ in law as $\e \ra 0$. 
 Then for all $x\in \RR^d$  $\mathcal{G}^{0, x} (g)$ is a limit point in law  of 
 $\mathcal{G}^{\e, x}(\e N^{\frac{\varphi_{\e}}{\e} })$ in $\DD([0,T], \RR^d)$.
 \end{proposition}
 
 \begin{proof}
 We drop the dependence on $x$ of $\ti X^{\e, x}$ and  
for every $\e\in (0, 1]$ and $t \in [0,T]$ we define  
\begin{align} \label{eq: the J and M}
 J_t^\e &:= \int_0^t  b(\tilde X^{\e}_s) ds + \int_0^t \int_{\RR^d \backslash \{0\}}  G(\tilde X^{\e}_{s}) (\varphi_{\e}(s,z) -1) z \nu(dz) ds, \nonumber \\
 \bar M_t^\e &:= \e \int_0^t \int_{\RR^d \backslash \{0\}} G(\tilde X^{\e}_{s-}) z \tilde{N}^{\frac{\varphi_{\e}}{\e}} (ds,dz).
 \end{align}
\paragraph{Step 1.}
We start by showing that the family of processes $(J^\e)_{\e\in (0, 1]}$ is C-tight 
(cf. \cite{Jacod Shiryaev}-Definition VI.3.25).  Let $\rho>0$. 
For every $\varepsilon>0$ let $(\rR(\varepsilon))_{\varepsilon>0}$ be given as in the statement of Proposition \ref{proposition: priori estimate controlled process}. 
Proposition \ref{proposition: priori estimate controlled process} yields $\varepsilon_0>0$ 
such that  $\varepsilon< \varepsilon_0$ implies 
\begin{align} \label{eq: LDP for Ctightness-1}
\bar \PP \Big ( \displaystyle \sup_{t \in [0, T]} |\tilde X^\varepsilon_t| > \rR(\varepsilon)\Big ) < \frac{\rho}{2}.
\end{align}
On the event $\{ \tilde \tau^\varepsilon_{\rR(\varepsilon)} >T\}$ there exists $C_1>0$ such that for any $s \lqq t \lqq T$ we have
\begin{align} \label{eq: LDP for Ctightness-2}
|J^\varepsilon_t - J^\varepsilon_s| \lqq  C_1\Big  (1+ \displaystyle \sup_{u \in [0, \tilde \tau^\varepsilon_{\rR(\varepsilon)}]} |\tilde X^\varepsilon_u| \Big  ) \Big ( (t-s) + \int_s^t \int_{\RR^d \backslash \{0\}} |z||\varphi^\varepsilon(u,z)-1| \nu(dz)du \Big ).
\end{align}
By Proposition \ref{proposition: a priori bound} let $\Lambda>0$ and $\varepsilon_1>0$ such that $\varepsilon < \varepsilon_1$ implies
\begin{align} \label{eq: LDP for Ctightness-3}
\bar \EE \Big [ \displaystyle \sup_{t \in [0,  \tilde \tau^\varepsilon_{\rR(\varepsilon)}]} |\tilde X^\varepsilon_t|^2 \Big ] \lqq \Lambda < \infty.
\end{align}
Estimate (\ref{eq: integrability controls- limit zero}) in Lemma \ref{lemma: integrability controls}  yields some $\delta_1>0$ such that for any $s \lqq t \lqq T$ with $|t-s| \lqq \delta_1$  it follows
\begin{align*}
 \displaystyle \sup_{g \in \mathfrak{S}^M} \int_s^t \int_{\RR^d \backslash \{0\}} |z||g(u,z)-1| \nu(dz) du < \frac{\rho^2}{4(C_1+ \Lambda)}.
\end{align*} 
Fix $\delta = \delta_\rho < \min \{\frac{\rho^2}{4 (C_1+ \Lambda)}, \delta_1 \}$. 
Due to (\ref{eq: LDP for Ctightness-1}), (\ref{eq: LDP for Ctightness-2}) and (\ref{eq: LDP for Ctightness-3}) 
$\varepsilon< \varepsilon_0 \wedge \varepsilon_1$ and Markov's inequality yield 
\begin{align} \label{eq: LDP for Ctightness-4}
 \bar \PP \Big (  |J^\varepsilon_t - J^\varepsilon_s|> \rho\Big ) & \lqq  \bar \PP \Big ( |J^\varepsilon_t - J^\varepsilon_s|> \rho; \tilde \tau^\varepsilon_{\rR(\varepsilon)} >T \Big ) + \bar \PP \Big ( \displaystyle \sup_{t\in [0, T]} |\tilde X^\varepsilon_t|> \rR(\varepsilon )\Big ) \nonumber \\
&\lqq \frac{1}{\rho} \bar \EE \Big [ \displaystyle \sup_{|t-s|< \delta} |J^\varepsilon_t - J^\varepsilon_s| \Big ] + \frac{\rho}{2}  \lqq  \rho.
\end{align} 
For $\rho>0$ and $\varepsilon< \varepsilon_0 \wedge \varepsilon_1$ we define the set 
\begin{align} \label{chpt2: eq Ktau for tightness}
K_{\rho} &:= \bigcap_{m \in \NN} K_{\rho,m}, \quad \text{where} \nonumber \\
K_{\rho,m}&:= \Big \{   f \in C([0,T], \RR^d)~\big|~ f(0)=0 \mbox{ and for all }s, t\in [0, T] \nonumber \\
& \quad \quad \quad |t-s|\lqq \delta_{\rho 2^{-m}} \mbox{ implies } |f(t)- f(s)| < \rho 2^{- m}\Big \} .
 \end{align}
 For every $m \in \NN$ the set $K_{\rho,m}$ is a non-empty collection of 
 equicontinuous and uniformly pointwise  bounded elements of $C([0,T];\RR^d)$. 
 Hence due to Arzela-Ascoli's theorem the set $K_{\rho,m}$ is relatively compact 
 in $C([0,T];\RR^d)$ for every $m \in \NN$. Since the non-empty intersection of relatively 
 compact sets in a metric space is relatively compact  we conclude that the set $K_\rho$ is a 
 non-empty relatively compact set in $C([0,T];\RR^d)$. Due to (\ref{eq: LDP for Ctightness-4}) it follows that
\begin{align*}
\bar \PP ( J^\e \notin K_{\rho} )  \lqq \rho \sum_{m=1}^{\infty} 2^{-m}  = \rho,
\end{align*}
which implies that $(J^\e)_{\e\in (0, 1]}$ is $C$-tight. 
\paragraph*{Step 2.}
We show that the family of processes $(\bar M^\e)_{\e>0}$ is C-tight as $\varepsilon \ra 0$. 
We fix the scale $(\rR(\e))_{\varepsilon>0}$, the constants $ \e_0, C>0$ given in 
Proposition \ref{proposition: priori estimate controlled process}  and 
\begin{align*}
C_2:= \displaystyle \sup_{g \in \mathfrak{S}^M} \int_0^T \int_{\RR^d \backslash \{0\}} |z|^2 g(s,z)\nu(dz) ds.
\end{align*}
Recall that $C_2 < \infty$ by (\ref{eq: integrability controls- first order}) in 
Lemma \ref{lemma: integrability controls}. Hence  for every $\kappa>0$ and any $\varepsilon< \varepsilon_0$ it follows
\begin{align*} \label{chpt2: eq para M no ldp} 
&\bar \PP \Big ([ \bar M^\e  ]_T > \kappa \Big ) 
\lqq \bar \PP \Big ( [ \bar M^\e  ]_T > \kappa, \sup_{s\in [0, T]} |\ti X^\e_s| \lqq \rR(\e) \Big ) + \bar \PP \Big (\sup_{s\in [0, T]} |\ti X^\e_s| > \rR(\e) \Big )\\
&\quad = \bar \PP \Big ( \e^2 \int_0^T \int_{\RR^d \backslash \{0\}} |G(\tilde X^{\e}_{s-})|^2 |z|^2 N^{\frac{\varphi_{\e}}{\e}} (ds,dz) > \kappa, \sup_{s\in [0, T]} |\ti X^\e_s| \lqq \rR(\e) \Big ) 
+2 e^{- \frac{1}{2} \rR(\e)} + C\e \rR(\e)\\
&\quad \lqq \bar \PP  \Big (  2 L^2 \e^2(1+ \rR^2(\varepsilon)) \int_0^T \int_{\RR^d \backslash \{0\}} |z|^{2} N^{\frac{\varphi_{\e}}{\e}} (ds,dz) > \kappa \Big ) +  2 e^{- \frac{1}{2} \rR(\e)} + C \e \rR(\e)\\
&\quad \lqq \frac{2 L^2\varepsilon(1+ \rR^2(\varepsilon))}{\kappa} \int_0^T \int_{\RR^d \backslash \{0\}} |z|^{2} \phi_\e(s, z) \nu(dz) ds   +2 e^{- \frac{1}{2} \rR(\e)} + C \e \rR(\e)\\
&\quad \lqq \frac{2 L^2C_2}{\kappa} (1+\rR^2(\e))\varepsilon  +2 e^{- \frac{1}{2}\rR(\e)} + C \e \rR(\e) \ra 0, \quad \mbox{ as } \e \ra 0, 
 \end{align*}
 since $\varepsilon \rR^2(\varepsilon) \ra 0$ whenever $\varepsilon \ra 0$.
In other words, $[ \bar M^\e  ]_T \ra 0$ as $\e \ra 0$ in probability and therefore in law, 
which implies that  $([ \bar M^\e])_{\e\in (0, 1]}$ is $C$-tight.  
\paragraph*{Step 3.}
Due to Theorem 6.1.1 in \cite{Kallianpur Xiong} the laws of the family  $\tilde Z_t^\e= x + J_t^\e+ M_t^\e$ 
are tight in $\DD([0,T], \mathbb{R}^d)$. By Prokhorov's Theorem there exists the weak limit 
of $(\tilde X^{\e_n}, J^{\e_n}, M^{\e_n})$ for some subsequence $\e_n \ra 0$. 
Skorokhod's representation's theorem implies that there exists a triplet of random variables 
$(\tilde X, \tilde \varphi ,0)$  defined on $(\bar{\mathfrak{M}}, \bB(\bar{\mathfrak{M}}), \bar \PP)$ 
such that $(\tilde X^{\e_n}, J^{\e_n}, M^{\e_n})$ given by (\ref{chpt2: eq controlled sde}) and (\ref{eq: the J and M})
converges to  $(\tilde X, \tilde \varphi,0)$ $\bar \PP$-a.s. as $n\ra \infty$.  
Due to  \eqref{eq: integrability controls- first order1} and the continuity 
of the functions $b$ and $G$ we can pass to the limit $\tilde X^{\e_n}_t \ra \tilde X_t$ 
pointwise (in $t \in [0,T]$) and  $\bar \PP$-a.s. in \eqref{chpt2: eq controlled sde}. 
Hence we have that $(\tilde X_s)_{t\in [0, T]}$ satisfies $\bar \PP$-a.s.  
\begin{align*}
\tilde X_t = x + \int_0^t b(\tilde X_s) ds + \int_0^t \int_{\RR^d \backslash \{0\}} G(\tilde X_s) 
(\tilde{\varphi}(s,z)-1) z \nu(dz) ds, \quad t\in [0, T].
\end{align*}
Therefore we conclude that $\tilde{X}= \mathcal{G}^{0}(\tilde \varphi)$. 
Combining that $\varphi $  and $\tilde \varphi$ have the same law under $\bar \PP$ and 
the C-tightness of $(\tilde X^\varepsilon)_{\varepsilon>0}$  implies the
$\bar \PP$-almost sure convergence $\displaystyle \sup_{t \in [0, T]} |\tilde X^{\varepsilon_n}_t - \tilde X_t| \ra 0$ 
as $n \ra \infty$ and hence the convergence in law we infer 
\[ 
\mathcal{G}^0 (\varphi) \quad \text{ is a weak limit point of } \quad \mathcal{G}^{\e} (\e N^{\frac{1}{\e} \varphi_{\e}}). 
\]
This finishes the proof.
\end{proof}

\paragraph*{Proof of Theorem \ref{thm: LDP full process}.}
 Proposition \ref{proposition: first condition LDP big jumps} and \ref{proposition: second part ldp big jumps process} imply
Condition 2.2(a) and (b) given  in \cite{BCD13} for $(X^{\e, x})_{\varepsilon>0}$. 
Hence Theorem 2.4 of \cite{BCD13} finishes the proof.
\begin{flushright} $\square$ \end{flushright}
\subsection{Some useful consequences}
\no In the sequel we establish the continuity of the LDP of $(X^{\e, x})_{\e>0}$ 
with respect  to the initial condition $x \in \RR^d$. 

\begin{proposition} \label{chpt2: proposition uniform large dev principle} 
 Given $T>0$ and $x \in D$ let $F \subset \DD([0,T], \mathbb{R}^d)$ 
 be closed and $G \subset \DD([0,T], \mathbb{R}^d)$ open with respect to the 
 Skorokhod topology. Then we have
\begin{align}
&\limsup_{\substack{\e\rightarrow 0\\ y \ra x}}  \e \ln  \bar \PP (X^{ \e, y} \in F) \lqq - \inf_{f \in F} \JJ_{x, T}(f), \label{eq: uniform ldp a}\\
&\liminf_{\substack{\e\rightarrow 0\\ y \ra x}}  \e \ln \bar  \PP (X^{ \e, y} \in G) \gqq - \inf_{g \in G} \JJ_{x, T}(g). \label{eq: uniform ldp b}
\end{align}
\end{proposition}
\begin{proof}
Due to Theorem 4.4 in \cite{Maroulas} the result follows from verifying the following statements.
\begin{itemize}
\item[1.] Let $(x_n)_{n \in \NN} \subset \RR^d$ such that $x_n \ra x$ as $n \ra \infty$. 
Given $M>0$ and $(g_n)_{n \in \NN} \in \mathfrak{S}^M$ such that $\nu^{g_n}_T \ra \nu_T^{g}$ 
in the vague topology as $n \ra \infty$. Then we obtain  
\begin{align*}
\gG^{0,x_n}(g_n) \ra \gG^{0,x}(g), \mbox{ as } n \ra \infty. 
\end{align*}
\item[2.] Let $M>0$, $(x_\varepsilon)_{\varepsilon>0} \subset \RR^d$ and 
$(\varphi_\varepsilon)_{\varepsilon>0} \subset \mathfrak{U}^M_{+}$ such that $x_\varepsilon \ra x$ 
and $\varphi_\varepsilon \Rightarrow \varphi$ in law  as $\varepsilon \ra 0$. 
Then we obtain the following convergence in law 
\begin{align*}
\gG^{\varepsilon, x_\varepsilon} \Big ( \varepsilon \tilde N^{\frac{1}{\varepsilon} \varphi_\varepsilon}\Big ) \Rightarrow \gG^{0,x}(g) \quad \text{ as } \varepsilon \ra 0.
\end{align*}
\end{itemize}
The verification of the conditions above is analogous to the proof of 
Theorem \ref{thm: LDP full process} and we omit its details.
\end{proof}

\no As a consequence of Proposition \ref{chpt2: proposition uniform large dev principle} 
we derive a uniform LDP for 
$(X^{\e,x})_{\e>0}$ when the initial state $x \in K$ for $K \subset D$ a 
closed (and bounded) set. The proof is virtually the 
same as the one given in the Brownian case and we omit it. 
We refer the reader to Corollary 5.6.15 in \cite{DZ98}.

\begin{corollary}\label{chpt2: corol Ldp uniform in compact sets of initial states} 
Let $T>0$, $K \subset D$ be compact, $F \subset \DD ([0,T], \mathbb{R}^d)$ closed, 
$G \subset \DD ([0,T], \mathbb{R}^d)$ open  with respect to  the $J_1$ topology and $x \in D$. 
Then it follows
\begin{align*}
& \limsup_{\e\rightarrow 0} \sup_{y \in K} \e \ln \bar  \PP (X^{\e, y} \in F) \lqq - \inf_{y \in K, f \in F} \JJ_{y, T}(f), \\
& \liminf_{\e\rightarrow 0} \inf_{y \in K} \e \ln \bar  \PP (X^{ \e, y} \in G) \gqq - \inf_{y \in K, g \in G} \JJ_{y, T}(g). 
\end{align*}
\end{corollary} 
\no In the sequel this result is applied to the first exit time problem of $X^{\e, x}$ from $D$. 

\bigskip
\section{The first exit time problem in the small noise limit} \label{sec: first exit}
In this section we fix the standing assumptions 
of the Hypotheses \ref{condition: det dynamical system}, \ref{condition: the measure nu}, \ref{condition: on the multiplicative coefficient} and \ref{condition: domain}
for some bounded domain $D \subset \RR^d$, $x\in D$ and $\nu \in \fM$.

\subsection{Continuity properties of the cost function} \label{sec: continuity pp}
 The following proposition ensures the (local) controllability 
of the dynamical system given by the controlled integral equation (\ref{eq: controlled ODE}) 
in small balls around the initial position. It plays a crucial role in the proof of 
the upper bound in Theorem \ref{thm: first exit time} given in the next subsection. 
We stress that in the more general setting discussed in Subsection \ref{sec: extensions remarks} 
the following result is stated as Hypothesis \ref{condition: generalized statement-potential} .
\begin{proposition} \label{condition: potential}  
Let Hypotheses \ref{condition: det dynamical system}, B, C and  \ref{condition: domain} be satisfied. 
For every $\rho_0>0$ there exist a constant 
$M>0$ and a non-decreasing function $\xi:[0, \rho_0] \ra \RR^{+}$ 
with $\lim_{\rho \ra 0} \xi(\rho)=0$ satisfying the following. 
Then for all $x_0, y_0 \in \RR^d$ such that $|x_0- y_0| \lqq \rho_0$ there exist 
$\Phi \in C([0, \xi(\rho_0)],\RR^d)$ and $g \in \mathfrak{S}^M$ 
such that $\Phi(\xi(\rho_0))=y_0$ and 
\begin{align} \label{eq: cond on potential: controllability cond C1}
 \Phi(t) = x_0 + \int_0^t b(\Phi(s)) ds + \int_0^t \int_{\RR^d \backslash \{ 0\}} G(\Phi(s)) (g(s,z)-1) z \nu(dz) ds, \quad t \in [0, \xi(\rho_0)].
\end{align}
\end{proposition}

\begin{proof}
For fixed $\rho_0>0$ and $x_0, y_0 \in \RR^d$ such that $|x_0-y_0|\lqq \rho_0$ consider the straight line that links $x_0$ and $y_0$,
\begin{align*}
\Phi(t):= x_0 + t \frac{y_0- x_0}{\rho_0}, \quad t \in [0, \rho_0].
\end{align*}
Let $\xi(\rho)= \rho$, $\rho \in [0, \rho_0]$. We observe that $\Phi \in C([0, \xi(\rho_0)]; \RR^d)$ 
and $\Phi(\xi(\rho_0))= y_0$. The construction of the control function $g \in \mathfrak{S}$ such that 
(\ref{eq: cond on potential: controllability cond C1}) holds follows from the next observation. 
Due to Hypothesis \ref{condition: the measure nu} every vector $x \in \RR^d$ can be written for 
some measurable function $f^x: \RR^d \longrightarrow [0,\infty)$ as
\begin{align*}
x= \int_{\text{supp}(\nu)} z f^x(z) \nu(dz).
\end{align*}
For instance we choose the function
\begin{align*}
f^x(z)= \frac{1}{\lambda^d(B_{R^{'}}(x))} \Big ( \frac{d \nu}{dz}(z) \Big )^{-1} \textbf{1}_{B_{R'}(x)}(z),
\end{align*}
where $\lambda^d$ is the Lebesgue measure on $(\RR^d, \bB(\RR^d))$ and $R':= \frac{1}{2} d(x,  (\text{supp}(\nu))^c )\wedge 1 $ 
with the convention of $d(x,\emptyset)=\infty$. Let $P_{x_0,y_0}(s)= \frac{y_0-x_0}{\rho_0} - b(\Phi(s))$ for every $s \in [0, \xi(\rho_0)]$. Set 
\begin{align*}
&g:[0, \xi(\rho_0)] \times \RR^d \backslash \{0\} \longrightarrow [0,\infty) \\
g(s,z)&:= 1+ \frac{\Big ( \frac{d \nu}{dz}(z) \Big )^{-1} }{\lambda^d(B_{R'}(P_{x_0,y_0}(s) ) ) } \frac{1}{G(\Phi(s))} \textbf{1}_{B_{R'}(P_{x_0,y_0}(s) )}.
\end{align*}
Since $\nu$ is a finite measure and $g$ is bounded it follows that $g \in \mathfrak{S}$ and (\ref{eq: cond on potential: controllability cond C1}) holds. This finishes the proof.
\end{proof}

\no We define the following cost function associated to the system \eqref{eq: the sde full perturbation}  
which measures the cost of steering $U^g$ given in \eqref{eq: controlled ODE} 
from its initial position $x\in D$ to some point $y\in \RR^d$ in exactly time $t>0$ by
\begin{align*}
V(x,y,t) := \inf \{ \eE_t(g) ~|~ g \in \mathfrak{S}_\varphi \quad \varphi(s) = U^g(s,x),\quad s \in [0,t], \quad \varphi(t)= y\}.
\end{align*}
The following continuity properties are essentially a consequence 
of Proposition \ref{condition: potential} and are shown in Subsection \ref{subsec: continuity cost function} of the appendix.  
\begin{lem} \label{prop: continuity pps cost functional} 
Let the assumptions of Theorem \ref{thm: first exit time} be satisfied. Then for any $\delta > 0$ there exists $\rho > 0$ such that 
\begin{align}
&\displaystyle \sup_{x, y \in B_{\rho}(0)} \displaystyle \inf_{t \in [0,1]} V(x,y, t) < \delta \label{eq: pp cont1} \\
&\displaystyle \sup_{\substack{x,y \in D \\ \inf_{z \in D^c} |x-z| + |y-z| \lqq \rho} } \displaystyle \inf_{t \in [0,1]} V(x,y,t) < \delta. \label{eq: pp cont2} 
\end{align}
\end{lem}
\begin{corollary} \label{corollary: finite height}
Let the assumptions of Theorem \ref{thm: first exit time} be satisfied. Then $\bar V < \infty$ for $\bar V$ being defined by (\ref{eq: the potential}).
\end{corollary}
\begin{proof}
We fix $z \in D^c$ and take $\rho_0=|z|$. By Proposition \ref{condition: potential} 
let $M< \infty$, $\xi:[0, \rho_0] \longrightarrow \RR^{+}$ 
and $g \in \mathfrak{S}^M$  such that \eqref{eq: cond on potential: controllability cond C1} 
holds for some $\Phi \in C([0, \xi(\rho_0)], \RR^d)$ and with $\Phi(\xi(\rho_0))=z$. Therefore we have  
\begin{align*}
\bar V := \displaystyle \inf_{z \in D^c} V(0,z) \lqq \int_0^{\xi(\rho_0)} \int_{\RR^d \backslash \{0\}} \ell(g(s,z)) \nu(dz) ds \lqq M < \infty.
\end{align*}
\end{proof}

\bigskip

\subsection{Proof of Theorem \ref{thm: first exit time}} \label{sec: the upper bound}
\begin{lemma} \label{chpt2: lemma first lemma upper bound} Let $c > 0$. 
There exist  $\rho_0 > 0$ and $s_0>0$ 
such that for any $\rho \in (0, \rho_0]$ it holds the limit   
\begin{align*}
\liminf_{\e \rightarrow 0} \e \ln \displaystyle \inf_{x \in B_\rho(0)} \bar  \PP (\sigma^{\e} (x) \lqq s_0) > - (\bar{V}+ c),
\end{align*}
where the potential height $\bar V$ is given in equation \eqref{eq: the potential}.
\end{lemma}    

\begin{proof} 
Let $\rho_0 > 0$ be small enough such that 
the inequalities of (\ref{eq: pp cont1}) and (\ref{eq: pp cont2}) in Lemma \ref{prop: continuity pps cost functional} 
are satisfied for $\delta = \frac{c}{2}$ and $\rho \lqq \rho_0$. Hence we may 
choose $x \in B_{\rho}(0)$ and a path 
 $\phi_1^x \in C([0, s_x], \mathbb{R}^d)$ satisfying $\phi_1^x(0) = x$, $\phi_1^x (s_x)= 0$ such that  
\begin{align*}
\JJ_{x, s_x} (\phi_1^x) \lqq \dfrac{c}{2}.
\end{align*}
With the help of (\ref{eq: pp cont2}) in Lemma \ref{prop: continuity pps cost functional} 
and Proposition \ref{condition: potential} 
we may choose $z \in  D^c \cap \text{supp}(\nu)$, $s_z > 0$, $\phi^z_2 \in C([0,s_z], \mathbb{R}^d)$ 
such that $\phi_2^z(0)=0$, $\phi_2^z(s_z)= z$ and 
    \begin{align*}
\JJ_{0,s_z} (\phi^z_2) \lqq \bar{V} + \dfrac{c}{2}.
\end{align*}
Let $\phi_3$ be the solution of the differential equation $\dot{\phi_3}= b(\phi_3)$ with $\phi_3(0)= z$. 
We set  $s_0=s_x+s_z + \delta'$ with $\delta'>0$ such that $\phi_3([0, \delta']) \subset D^c$ and define 
\begin{align*}
 \Phi^x(t) := \begin{cases} \phi_1^x(t)  & \text{ if } t\in [0,  s_x],\\
 \phi_2^z(t- s_x) &  \text{ if } t\in (s_x, s_z + s_x], \\
 \phi_3 (t-s_z -s_x) &  \text{ if }  t\in (s_x+ s_z, s_0]. \end{cases} \end{align*}
 Then the concatenation of the paths yields 
 \begin{align*}
 \JJ_{x, s_0} (\Phi^x) \lqq \JJ_{x, s_x} (\phi_1^x) + \JJ_{0,s_z}(\phi^z_2) +0 \lqq \bar{V} + c.
 \end{align*}
 Let $\Delta = d(z, \bar{D})$ and consider the open set
 \begin{align*}
 \mathcal{O}= \bigcup_{x\in B_{\rho_0}(0)} \{ \psi \in \DD ([0, s_0], \mathbb{R}^d) ~|~ d_{J_1}( \psi , \Phi^x ) < \dfrac{\Delta}{2}  \}.
 \end{align*}
 The constructed path $\Phi^x$ visits $z$ by definition and stays outside of $D$ in the time interval $[s_x+s_z, s_0]$, 
 due to the choice of $z \in D^c $ and the continuity of $\varphi_3$. 
 By definition of $\mathcal{O}$ every path $\psi \in \mathcal{O}$ exits $D$ before time $s_0$. 
 We show this claim by contradiction. Fix $\psi \in \oO$. Let us suppose that $\psi([0,s_0]) \subset D$. 
 This implies that 
 \begin{align} \label{chpt2: eq upper bound contradiction } 
 d(z, \overline{\psi([0,s_0])} )> \Delta.
 \end{align}
  Since $\psi \in \oO$ we have $d_{J_1}(\psi, \Phi^x) < \frac{\Delta}{2}$, that is,  
  there is an increasing homeomorphism $\la: [0,s_0] \ra [0,s_0]$ such that 
 \begin{align*}
 \displaystyle \sup_{t \in [0,s_0]} |\psi(\lambda(t)) - \Phi^x(t)| < \frac{\Delta}{2}.
 \end{align*}
 In particular, $$|\psi(\lambda(s_z + s_x)) - \Phi^x(s_z+ s_x)|= |\psi(\lambda(s_z + s_x)) -z| < \frac{\Delta}{2},$$
 which contradicts \eqref{chpt2: eq upper bound contradiction }. 
 Corollary \ref{chpt2: corol Ldp uniform in compact sets of initial states} yields 
 \begin{align*}
 \liminf_{\e \rightarrow 0} \e \ln \inf_{x\in B_{\rho_0}(0)} \bar \PP (\sigma^\e (x) \lqq s_0) 
 & \gqq \liminf_{\e \rightarrow 0} \e \ln \inf_{x\in B_{\rho_0}(0)} \bar \PP (X^{ \e, x}  \in \mathcal{O} ) \\
 & \gqq  - \sup_{x\in B_{\rho_0}(0) } \inf_{\psi \in \mathcal{O}} \JJ_{x, s_0} (\psi) 
 \gqq -\sup_{x\in B_{\rho_0}(0)} \JJ_{x, s_0} (\Phi^x) \gqq - (\bar{V} + c), 
 \end{align*}
 which finishes the proof.
\end{proof}
\no For fixed $x \in D$ and small $\varepsilon>0$ we show that the probability of $X^{\e,x}$ staying in $D$ in the long run 
 without hitting a small neighborhood 
of $0$ is exponentially negligible. For given $\rho>0$ such that $\bar B_{\rho} (0) \subset D$, we define 
\begin{align} \label{eq: auxiliary first exit time big jumps}
\vt_{\rho}^\e (x) := \inf \{ t \gqq 0~\big|~  |X^{ \e, x}_t| \lqq \rho \text{ or } X^{\e, x}_t \in D^c \}.
\end{align}  
\begin{lemma} \label{chpt2: lemma second lemma upper bound} We have
\begin{align}
\lim_{t \rightarrow \infty} \limsup_{\e \rightarrow 0} \e \ln \sup_{x \in D} \bar \PP (\vartheta^\e_{\rho} (x) > t) = - \infty. 
\label{eq: long nonexit is exp neg}
\end{align}
\end{lemma}
\begin{proof}
Let us fix $\rho>0$. For $t \gqq 0$ we define the subsets of $\DD([0,t], \mathbb{R}^d)$ 
\begin{align*}
\mathcal{G}_t &:= \Big  \{ \Phi \in  \DD ([0,t], \mathbb{R}^d)~|~ \Phi(s) \in \overline{D \backslash B_\rho (0)} \quad \text{for all } s \in [0,t] \Big \} \mbox{ and } \\
 \tilde \gG_t &:=  \Big \{  \Phi \in  \DD ([0,t], \mathbb{R}^d)~|~ \Phi(s) \in \overline{D \backslash  B_\rho (0)} \text{ for all  } s \in [0,t]  \\ 
 &\qquad  \text{ except in a countable number of points} \Big  \}.
\end{align*}
In Lemma \ref{claim: first topological claim} in Subsection \ref{subsec: topology} of the appendix 
it is shown that due to right continuity we have $\tilde \gG_t = \gG_t$ and $\tilde \gG_t$ is a closed 
set in $\DD([0,t], \RR^d)$ with respect to the Skorokhod topology. 
By the definition of $\gG_t$ and Corollary~\ref{chpt2: corol Ldp uniform in compact sets of initial states} we have
\begin{align}
\limsup_{\e \rightarrow 0} \e \ln \sup_{x \in D} \bar \PP (\vt^\e_{\rho} (x)  > t) 
& \lqq \limsup_{\e \rightarrow 0} \e \ln \sup_{  x \in \overline{D\backslash B_\rho (0) }}  \bar \PP (\vt^\e_{\rho} (x)  > t) 
\lqq -\inf_{  x \in \overline{D\backslash B_\rho (0)}} \displaystyle \inf_{\psi \in \tilde \gG_t} \JJ_{x, t} (\psi) \non\\
&= -\inf_{  x \in \overline{D\backslash B_\rho (0)}} \displaystyle \inf_{\psi \in \mathcal{G}_t} \JJ_{x, t} (\psi) 
= - \inf_{\psi \in \mathcal{G}_t} \JJ_{\psi (0), t} (\psi).\label{eq: step 1}
\end{align}
\paragraph*{Claim:} We have 
\begin{align}
\lim_{t \rightarrow  \infty} \inf_{\psi \in \mathcal{G}_t} \JJ_{\psi(0),t} (\psi) =  \infty. \label{eq: step2}
\end{align}
Let $(\varphi_t)_{t \gqq 0}$ be the dynamical system associated to $\dot{\varphi}_t = b(\varphi_t)$ on $\RR^d$. 
Due to Hypothesis \ref{condition: det dynamical system} for any  $x \in D \backslash \bar B_\rho (0)$
there exists $t_x \gqq 0$ such that $\varphi(t_x) \in B_{\frac{\rho}{2}} (0)$. We define
the open neighborhood $O_x := \varphi^{-1} (B_{\frac{\rho}{2}} (0))$ of $x$ in $\RR^d$. 
By compactness there are $k \in \NN$ and  $x_1,\dots,x_k \in D \backslash \bar B_{\rho} (0)$ such that 
$\bigcup_{i=1}^{k} O_{x_i} \supset ( D \backslash \bar B_{\rho}(0)) $. 
We set $s= t_{x_1} \vee \dots \vee t_{x_k} $. 
Before time $s$ any path that solves $\dot{\varphi}_t = b(\varphi_t)$, 
with initial condition in $D \backslash \bar B_{\rho}(0)$ hits $B_{\frac{\rho}{2}}(0)$. 
We argue by contradiction. Assume that
\begin{align}
\lim_{t \rightarrow + \infty} \inf_{\psi \in \mathcal{G}_t} \JJ_{\psi(0),t} (\psi) < \infty. \label{eq: contradiction} 
\end{align}
Let us fix $M>0$ such that for any $n \in \NN$ there exists $\psi^n \in \gG_{ns}$ 
verifying $\JJ_{\psi^n (0), ns} (\psi^n) \lqq M$. 
For $k=0,\dots,n-1$ let
\[ 
\psi^{n,k} (t) := \psi^n (k\cdot (s-t)), t \in [0,s].
\]
Hence $\psi_{n,k} \in \mathcal{G}_s$ and
\begin{align} \label{eq: estimate for lsc argument}
M \gqq \JJ_{\psi^n (0), ns} (\psi^n) = \sum_{i=0}^{n-1} \mathbb{J}_{ \psi^n (ks), s} (\psi^{n,k}) 
\gqq n \displaystyle \min_{0 \lqq k \lqq n-1}  \mathbb{J}_{ \psi^{n,k} (0), s} (\psi^{n,k}).
\end{align}
We finally show the existence of a sequence $(\psi^n)_{n \in \NN}$ in $\mathcal{G}_t$ such that
\[ \displaystyle \lim_{n \rightarrow \infty} \mathbb{J}_{\psi^n (0), s} (\psi^n) = 0.\]
First we see that the set 
$$
\{ \psi \in \DD([0,s], \mathbb{R}^d) ~|~\psi(0) \in \overline{D\backslash B_{\rho}(0)}, 
\mathbb{J}_{\psi(0), s} (\psi) \lqq 1 \}
$$ 
is a closed subset of the compact set 
$\{ \psi \in \DD([0,s], \mathbb{R}^d) ~|~\mathbb{J}_{\psi(0), s} (\psi) \lqq 1 \}$. 
The compactness comes from the fact that $\JJ_{\psi(0), s}$ is a good rate function with respect 
to the Skorokhod topology. 
Hence the sequence $(\psi^n)_{n \in \NN}$ has a limit point 
in $\mathcal{G}_s$ which we call $\bar{\psi}$. 
Since $\mathbb{J}_{\psi(0), s}= \displaystyle \inf_{x \in \mathbb{R}^d} \mathbb{J}_{x,s} $ is lower 
semicontinuous and due to (\ref{eq: estimate for lsc argument}) it follows that 
$\mathbb{J}_{\bar{\psi} (0), s} (\bar{\psi})= 0$. Due to the definition of rate function in 
(\ref{eq: rate function}), the structure of the controlled paths in (\ref{eq: controlled ODE}) 
and (\ref{eq: entropy functional}) this implies that $\bar{\psi}$ solves 
$\dot{\bar \psi}_t = b(\bar \psi_t)$ with $\bar{\psi} (0) \in D \backslash \bar B_{\rho} (0)$. 
Therefore $\bar{\psi}$ reaches $B_{\frac{\rho}{2}} (0)$ before time $s$, 
which contradicts $\bar{\psi} \in \mathcal{G}_s$ and thus assumption 
\eqref{eq: contradiction}. Combining inequality 
\eqref{eq: step 1} and \eqref{eq: step2} yields the desired result \eqref{eq: long nonexit is exp neg}. 
\end{proof}

\begin{theorem} \label{chpt2: thm upper bound first exit time}
For $x \in D$ and $\delta > 0$  we have
\begin{align}\label{eq: exit time upper bound}
&\liminf_{\e \rightarrow 0} \e \ln \bar  \PP (\sigma^\e (x) < e^{\frac{\bar{V}+ \delta}{\e}}) = 1 \text{ and } \nonumber \\
& \limsup_{\varepsilon \ra 0} \varepsilon \ln \bar \EE [\sigma^\varepsilon(x)] \lqq \bar V + \frac{\delta}{2}.
\end{align}
\end{theorem}
\begin{proof} The proof consists of two steps. 
\begin{claim} \label{chpt2: claim upper bound} 
For any $\delta > 0$ there are $T > 0$, $c>0$ and $\e_0 \in (0, 1] $ such that $\e \in (0, \e_0]$ implies 
\begin{align*}
\inf_{x \in D} \bar \PP (\sigma^{\e} (x) \lqq T) \gqq c e^{- \frac{\bar{V}+ \frac{\delta}{2}}{\e}}.
\end{align*}
\end{claim}
\no We first observe that by Lemma  \ref{chpt2: lemma first lemma upper bound} 
for every $\delta>0$ there are $t_0 > 0$ and $\rho> 0$ such that
\begin{align*}
\liminf_{\e \rightarrow 0} \e \ln \inf_{x\in B_{\rho}(0)} \bar  \PP (\sigma^{\e} (x) \lqq t_0) > - (\bar{V} + \frac{\delta}{4}). 
\end{align*}
For the fixed value $\rho>0$ and any $r >0$ Lemma \ref{chpt2: lemma second lemma upper bound} yields $t_1 > 0$ and $\varepsilon_0 \in (0,1]$ such that  $\varepsilon < \varepsilon_0$ implies
\begin{align*}
\e \ln \sup_{x \in D} \bar  \PP (\vt^\e_\rho(x) > t_1) < -r.
\end{align*}
In addition, let $\tilde c>0$ and $\e_0\in (0, 1]$ sufficiently small 
such that for $\e \in (0,  \e_0]$ it follows $1 - e^{- \frac{r}{\e}} > \tilde c e^{- \frac{\delta}{4 \e}}$. 
Since  $\{  \vt^\e_\rho(x) < \sigma^{\e} (x) \}= \{ X^{ \e, x}_{ \vt^\e_\rho(x) } \in \bar B_{\rho}(0) \}$ we have on this event
\begin{align*}
\sigma^{\e} (x)  = \vt^\e_\rho(x) + \sigma^{\e} (X^{ \e, x}_{\vt^\e_\rho(x)}) \circ \Theta_{\vt^\e_\rho(x) }, 
\end{align*}
where $\Theta_{s}$ is  the canonical shift  by time $s$ on the path space $\DD([0, \infty), \RR^d)$. 
Using the homogeneous strong Markov property of $X^{ \e, x}$ we obtain  
for any fixed $\e \in(0, \e_0]$ and $x \in D$ 
\begin{align*}
\bar \PP (\sigma^{\e} (x)\lqq t_0 + t_1) & \gqq  \bar \PP\big (\vt^\e_{\rho} (x) 
\lqq t_1 \text{ and } \sigma^{\e}(X^{ \e, x}_{\vt^\e_{\rho} (x) })  \lqq t_0 \big ) \\
& = \bar  \PP \big (\vt^\e_{\rho} (x) \lqq t_1\big ) \bar \PP \big (  \sigma^{\e}(X^{ \e, x}_{\vt^\e_{\rho} (x)}) 
\lqq t_0 | \vt^\e_{\rho} (x) \lqq t_1 \big ) \\
& \gqq \inf_{y \in D} \bar \PP (\vt^\e_{\rho} (y) \lqq t_1) \inf_{x\in B_{\rho}(0)} \bar  \PP (\sigma^{\e} (x) \lqq t_0) \\
& \gqq  c e^{- \frac{\bar{V} + \frac{\delta}{4}}{\e}} \tilde c e^{- \frac{\delta}{4 \e}} 
\gqq  c \tilde c e^{- \frac{\bar{V} + \frac{\delta}{4}}{\e}}  (1 - e^{- \frac{r}{\e}})  = c \tilde c e^{- \frac{\bar{V}+ \frac{\delta}{2}}{\e}}.
\end{align*}
\no Setting $T=t_0 + t_1$ and renaming the constants we finish the proof of Claim \ref{chpt2: claim upper bound}. 
\medskip
\paragraph{Step 2: } We continue with the proof of the limit \eqref{eq: exit time upper bound} and set $q^\e := \inf_{x \in D} \bar  \PP (\sigma^{\e} (x)\lqq T)$ for the time $T>0$ given in Claim \ref{chpt2: claim upper bound}.  
Claim \ref{chpt2: claim upper bound} yields $q^\e>0$ for all $\e \in (0, \e_0]$. 
For any  $k \in \NN$ and $x \in D$ we consider 
the family of events $\{ \sigma^{\e} (x) > kT\}$ for which we derive the following recursion
\begin{align*}
 \bar \PP \big  (\sigma^{\e} (x) > (k + 1)T \big ) 
 &=\Big  (1-  \bar \PP \big (  \sigma^{\e} (x)\lqq (k+1)T | \sigma^{\e} (x) > kT \big ) \Big  ) \cdot
\bar \PP \big  (\sigma^{\e} (x) > kT \big ) \\
& \lqq (1-q^\e) \cdot \bar \PP \big (\sigma^{\e} (x) > kT\big ), \quad k \in \NN.
\end{align*}
Solving the recursion above in $k \in \NN$ we obtain for any $\e \in (0, \e_0]$ 
\begin{align*}
\sup_{x \in D} \bar  \PP (\sigma^{\e} (x) > kT) \lqq (1-q^\e)^k, \quad k \in \NN.
\end{align*}
This implies the following bound 
\begin{align*}
\sup_{x \in D} \bar \EE [\sigma^{\e} (x) ]  &=  \displaystyle \sup_{x \in D} T \int_0^\infty  \bar \PP (\sigma^\e(x) > Ts) ds 
\lqq T \displaystyle \sup_{x \in D}  \sum_{k= 0}^{\infty}  \bar  \PP (\sigma^{\e} (x) > kT) \lqq  T \sum_{k=0}^{\infty} (1-q^\e)^k = \frac{T}{q^\e}.
\end{align*}
Since we have $q^\e \gqq e^{- \frac{\bar{V}+ \frac{\delta}{2}}{\e}}$ for $\e \in (0, \e_0]$ we obtain 
\begin{align*}
\sup_{x \in D} \bar \EE [\sigma^{\e} (x)] \lqq T e^{\frac{\bar{V}+ \frac{\delta}{2}}{\e}}.
\end{align*}
Chebyshev's inequality implies for all $x \in D$ and $\e \in (0, \e_0]$,
\begin{align} \label{eq: to use in tthe first exit time full perturbation1}
 \bar \PP ( \sigma^{\e} (x)  \gqq e^{\frac{\bar{V}+ \delta}{\e}}) 
 \lqq e^{-\frac{\bar{V}+ \delta}{\e}} \bar \EE [\sigma^{\e} (x)] 
 \lqq  e^{- \frac{\delta}{2 \e}}.
 \end{align}
 Sending $\e \rightarrow 0$ we conclude.
\end{proof}

\bigskip

\begin{lemma} \label{chpt2: lemma first lemma in lower bound first exit time} 
 For any $x \in D$ and $\rho > 0$ such that $ \bar{B}_\rho (0) \subset D$ we have
\begin{align*}
\lim_{\e \rightarrow 0} \bar \PP (X_{\vt^\e_{\rho}(x) }^{\e, x} \in \bar{B}_{\rho} (0)) = 1.
\end{align*}
\end{lemma}

\begin{proof} 
We fix $\rho>0$ and  $x \in D \backslash \bar B_{\rho} (0).$ Otherwise the result is trivial. 
Due to Hypothesis \ref{condition: det dynamical system}.1 there exists $T>0$ such that 
$X^{0,x}_t \in B_{\frac{\rho}{2}}$ for all $t \gqq T$. Hypothesis \ref{condition: domain} yields
\begin{align*}
\Delta := \rho \wedge \text{dist} \Big ( \{X^{0,x}_t ~|~ t \in [0,T] \}, D^c \Big ) >0.
\end{align*}
Hence it follows that 
\begin{align*}
\Big \{ 
X^{\e,x}_{\vt^\e_{\rho}(x)} \in D^c \Big \} 
&\subset \Big \{ \displaystyle \sup_{t\in [0, T \wedge \vartheta^\e_\rho(x)]} |X^{ \e, x}_t - X^{0,x}_t | > \frac{\Delta}{2} \Big \} \\
&\subset 
\Big \{ \displaystyle \sup_{t\in [0, T \wedge \vartheta^\e_\rho(x)]} |X^{ \e, x}_t - \bar\EE[X^{\e,x}_t] | > \frac{\Delta}{4} \Big \} 
\cup \Big \{ \displaystyle \sup_{t\in [0, T \wedge \vartheta^\e_\rho(x)]} 
|\bar\EE[X^{ \e, x}_t] - X^{0,x}_t | > \frac{\Delta}{4} \Big \}.
\end{align*}
We first show that for $\e$ sufficiently small the second event is empty. Indeed, 
\begin{align*}
|\bar\EE[X^{ \e, x}_t] - X^{0,x}_t |^2  
&= |\bar\EE[X^{ \e, x}_t - X^{0,x}_t] |^2 \lqq \bar\EE[|X^{ \e, x}_t - X^{0,x}_t|^2] \mbox{ for all }t\gqq 0,
\end{align*}
and by It\^o's formula we have for $t\lqq T$
\begin{align*}
\bar\EE[|X^{\e, x}_{t \wedge \vt^\e_\rho(x)}  -  X^{0,x}_{ \wedge \vt^\e_\rho(x)}|^2 ]
&= 2 \bar\EE\big[\int_0^{ \wedge \vt^\e_\rho(x)} \lgl X^{\e, x}_s -  X^{0,x}_s, b(X^{\e, x}_s) - b(X^{0,x}_s) \rgl ds\big] \\
&\qquad + \bar\EE\big[\int_0^{t \wedge \vt^\e_\rho(x)} |\e G(X^{\e,x}_{s-}, z)|^2 \frac{1}{\e} \nu(ds) ds\big] \\
&\lqq- 2c_1  \bar\EE\int_0^{ t\wedge \vt^\e_\rho(x)} |X^{\e, x}_s -  X^{0,x}_s|^2 ds] \\
&\qquad + \e \bar\EE[\int_0^{t \wedge \vt^\e_\rho(x)} \int_{\RR^d} L (1+|X^{\e,x}_{s-}, z)|^2) \nu(dz) ds].
\end{align*}
We  the  estimate 
\begin{align*}
\bar\EE[|X^{\e, x}_{t \wedge \vt^\e_\rho(x)} -  X^{0,x}_{t \wedge \vt^\e_\rho(x)}|^2 ]
&\lqq \e T L (1+\diam(D)^2)\int_{\RR^d} |z|^2 \nu(dz)
\end{align*}
and hence for all $\e < \frac{\Delta}{4} C$, $C= \big(T L (1+\diam(D)^2)\int_{\RR^d} |z|^2 \nu(dz)\big)^{-1}$ we have 
\begin{align*}
|\bar\EE[X^{ \e, x}_t] - x | \lqq \frac{\Delta}{4}. 
\end{align*}
Therefore it follows for any $\lambda>0$ that
\begin{align} \label{eq: lemma 17- initial estimate0}
&\bar \PP \Big ( X^{\e,x}_{\vt^\e_{\rho}(x)} \in D^c \Big ) 
\lqq \bar \PP \Big (  \displaystyle \sup_{t\in [0, T \wedge \vartheta^\e_\rho(x)]} |X^{ \e, x}_t - \bar\EE[X^{\e,x}_t] | > \frac{\Delta}{4} \Big ) \nonumber \\
& \lqq \bar \PP \big ( \displaystyle \sup_{t\in [0, T \wedge \vartheta^\e_\rho(x)]} |X^{ \e, x}_t - \bar\EE[X^{\e,x}_t] | > \frac{\Delta}{4}~|~ [X^{\e,x} - \bar\EE[X^{\e,x}]]_{T \wedge \vartheta_\rho^\e(x)} \lqq \lambda \big )\nonumber\\
&\qquad + \bar \PP \big ( [X^{\e,x}- \bar\EE[X^{\e,x}]]_{T \wedge \vartheta_\rho^\e(x)} > \lambda \big ).
\end{align}
 In this case the Bernstein-type inequality for local martingales given by Theorem 3.3 of \cite{DZ01} reads for $X^{\e, x}_t - \bar\EE[X^{\e, x}_t]$ as follows 
 \begin{align} \label{eq: lemma 17- bernstein0}
 \bar \PP \Big ( \displaystyle \sup_{t\in [0, T \wedge \vartheta^\e_\rho(x)]} |X^{ \e, x}_t - \bar\EE[X^{\e,x}_t] | > \frac{\rho}{2}~|~ [X^{\e,x} - \bar\EE[X^{\e,x}]]_{T \wedge \vartheta_\rho^\e(x)} \lqq \lambda \Big ) \lqq 2 \exp \Big ( - \frac{1}{2} \frac{\rho^2}{4 \lambda}\Big ).
 \end{align}
 Hypotheses \ref{condition: on the multiplicative coefficient} and \ref{condition: domain} yield some constant $C>0$ such that for $\e>0$ small enough 
 \begin{align*}
 [X^{\e,x}- \bar\EE[X^{\e,x}]]_{T \wedge \vartheta^\e_\rho(x)} \lqq C \e^2 \int_0^T \int_{\RR^d \backslash \{0\}} |z|^{2} N^{\frac{1}{\e}}(ds,dz).
 \end{align*}
 Therefore we obtain for $\e>0$ small enough 
 \begin{align} \label{eq: quadratic variation}
 \bar \PP \Big ( [X^{\e,x} - \bar\EE[X^{\e,x}]]_{T \wedge \vartheta^\e_\rho(x)} > \lambda \Big ) 
 &\lqq \frac{\e^2 C}{\lambda} \bar \EE \Big [ \int_0^T \int_{\RR^d \backslash \{0\}} |z|^{2} N^{\frac{1}{\e}}(ds,dz)\Big ] \lqq C C_\nu^{2} T\frac{\e}{\lambda} 
 \end{align}
where $C_\nu^{2}:= \int_{\RR^d \backslash \{0\}} |z|^{2} \nu(dz)< \infty$ due to the fact that $\nu$ is 
a L\'{e}vy measure respecting the integrability condition \eqref{eq: integrability cond measure}.
 Hence choosing $\lambda:=\lambda_\e:= \e^{\frac{1}{2}}$  the inequalities (\ref{eq: lemma 17- bernstein0}) 
 and (\ref{eq: quadratic variation}) imply with (\ref{eq: lemma 17- initial estimate0}) that
 \begin{align*}
 \bar \PP \Big ( X^{\e,x}_{\vartheta^\e_\rho(x)} \notin D \Big ) \lqq 2 \exp \Big (- \frac{1}{2} \frac{\Delta^2}{4 \sqrt{\e}} \Big ) + CT c_\nu^{2} \e^{\frac{1}{2}}.
 \end{align*}
Sending $\e \ra 0$ we infer the desired result. 
\end{proof}

\medskip 
\begin{lemma} \label{chpt2: lemma second lemma lower bound first exit time} 
For any $\rho > 0$ and $c>0$ we have 
\begin{align*}
\limsup_{\e \rightarrow 0} \e \ln \sup_{x \in D} \bar \PP ( \displaystyle \sup_{t\in [0,\e\wedge \sigma^\varepsilon(x)]} 
|X_t^{ \e,x} -x| \gqq \rho) < -c.
\end{align*}
\end{lemma}

\begin{proof}
We fix $\rho>0$ and  $x \in D \backslash \bar B_{\rho} (0)$. Otherwise the result is trivial. 
Let $p, q>0$ be the constants appearing in Hypothesis E.3. 
First note that due to the Hypothesis E.2 there is a constant $C>0$ such that 
for any scale $\theta_\e / \e \ra \infty$ as $\e \ra 0$ we have 
 \begin{align*}
 \bar\PP(\inf\{t> 0~|~ |\e \Delta_t L^\e|>\theta_\e\} \lqq \e)
 &= 
 1- \exp(- \e(1/\e) \nu((\theta_\e/\e) B_1^c(0)))\\
 &\lqq 1- \exp(- C e^{-\Gamma\frac{\theta_\e^2}{\e^2}}) 
    \lqq 2 C e^{-\Gamma\frac{\theta_\e^2}{\e^2}}. 
 \end{align*}
For the threshold $\theta_\e = \sqrt{\e} |\ln(\e)|^{q}$ 
this limit is asymptotically exponentially negligible as $\e \ra 0$. 
Hence we consider $X^{\e, x}$ to be conditioned to have 
jumps with size less or equal to $\sqrt{\e} |\ln(\e)|^{q}$. 
We denote this process by $\xX^{\e, x}$. In addition, we have  
\begin{align*}
\Big \{ 
\sup_{t\in [0, \e\wedge \sigma^\varepsilon(x)]} |\xX^{\e,x}_{t} -x| \gqq \rho \Big \} 
&\subset 
\Big \{ \displaystyle \sup_{t\in [0, \e\wedge \sigma^\varepsilon(x)]} |\xX^{ \e, x}_t - \bar\EE[\xX^{\e,x}_t] | \gqq \frac{\rho}{2} \Big \} 
\cup \Big \{ \displaystyle \sup_{t\in [0, \e]} 
|\bar\EE[\xX^{ \e, x}_t] - x | > \frac{\rho}{2} \Big \}.
\end{align*}
We first show that for $\e$ sufficiently small the second event is empty. For any $t>0$ we have 
\begin{align*}
|\bar \EE[\xX^{ \e, x}_t] -x |^2  
&= |\bar\EE[\xX^{ \e, x}_t - x] |^2 \lqq \bar \EE[|\xX^{ \e, x}_t - x|^2].
\end{align*}
We continue with the help of $b(0) = 0$ and It\^o's formula 
\begin{align*}
\bar \EE[|\xX^{\e, x}_t -  x|^2 ]
&= 2 \bar\EE\int_0^t \lgl \xX^{\e, x}_s , b(X^{\e, x}_s) ) \rgl ds 
+ \bar\EE\int_0^t \int_{|z|\lqq \sqrt{\e} |\ln(\e)|^{q}}|\e G(\xX^{\e,x}_{s-}, z)|^2 \frac{1}{\e} \nu(ds) ds \\
&\lqq- 2c_1  \bar\EE\int_0^t |\xX^{\e, x}_s|^2 ds 
+ \e \int_0^t \int_{|z|\lqq \sqrt{\e}|\ln(\e)|^{q}} L (1+\bar\EE[|\xX^{\e,x}_{s-}|^2 ]) |z|^2\nu(dz) ds\\
&\lqq \e L \int_{|z|\lqq 1}|z|^2\nu(dz) \int_0^t  (1+\bar\EE[|\xX^{\e,x}_{s-}|^2 ]) ds\\
&\lqq \e L \int_{|z|\lqq 1}|z|^2\nu(dz) \int_0^t  (1+2 |x|^2 + 2\bar\EE[|\xX^{\e,x}_{s-}-x|^2 ]) ds\\
&\lqq \e L \int_{|z|\lqq 1}|z|^2\nu(dz) \Big( t (1+2 |x|^2) 
+ 2 \int_0^t\bar \EE[|\xX^{\e,x}_{s-}-x|^2 ] ds\Big).
\end{align*}
Hence Gronwall's lemma yields for $t\in [0, \e]$, $\e \in (0, 1]$ and 
$C=  L \int_{|z|\lqq 1}|z|^2\nu(dz)$ a positive constant $C_{x}$ such that 
\begin{align*}
\bar \EE[|\xX^{\e, x}_t -  x|^2 ] \lqq \e C (1+2 |x|^2) \int_0^t s\exp(\e 2 C (t-s)) ds \lqq \e^3 C_{x}.  
\end{align*}
and hence for all positive $\e$ such that $\e^3  C_{x} < (\frac{\rho}{2})^2$ we have 
\begin{align*}
\sup_{t\in [0, \e]}|\bar \EE[\xX^{ \e, x}_t] - x| \lqq \frac{\rho}{2}. 
\end{align*}
Therefore it follows for any $\la>0$ 
\begin{align} \label{eq: lemma 17- initial estimate}
\bar \PP \Big (  \displaystyle \sup_{t\in [0, \e\wedge \sigma^\varepsilon(x)]} |\xX^{ \e, x}_t - x | \gqq \rho \Big ) 
&\lqq \bar \PP \Big (  \displaystyle \sup_{t\in [0, \e\wedge \sigma^\varepsilon(x)]} |\xX^{ \e, x}_t - \bar\EE[\xX^{ \e, x}_{t}] | \gqq \rho \Big ) \nonumber\\
& \lqq \bar \PP \big ( \displaystyle \sup_{t\in [0, \e\wedge \sigma^\varepsilon(x)]} |\xX^{ \e, x}_t - \bar\EE[\xX^{\e, x}_t | \gqq \frac{\rho}{2}~|~ [\xX^{\e,x} - \bar\EE[\xX^{\e,x}]]_{\e\wedge \sigma^\varepsilon(x)} \lqq \lambda \big )\nonumber\\
&\qquad + \bar \PP \big ( [\xX^{\e,x}- \bar\EE[\xX^{\e,x}]]_{\e\wedge \sigma^\varepsilon(x)} > \lambda \big ).
\end{align}
Theorem 3.3 of \cite{DZ01} yields the following for the local martingale $\xX^{\e, x}_t - \bar\EE[\xX^{\e, x}_t]$ 
 \begin{align} \label{eq: lemma 17- bernstein}
 \bar \PP \Big ( \displaystyle \sup_{t\in [0, \e\wedge \sigma^\varepsilon(x)]} |\xX^{ \e, x}_t - \bar\EE[\xX^{\e,x}_t] | > \frac{\rho}{2}~|~ [\xX^{\e,x} - \bar\EE[\xX^{\e,x}]]_{\e\wedge \sigma^\varepsilon(x)} \lqq \lambda \Big ) \lqq 2 \exp \Big ( - \frac{1}{2} \frac{\rho^2}{4 \lambda}\Big ).
 \end{align}
 For $\la = \la_\e := \frac{\e}{\sqrt{|\ln(\e)|^p}}$ the latter estimate is asymptotically exponentially neglibile as $\e \ra 0$. Hence it remains to prove that 
 \[
 \bar \PP \big ( [\xX^{\e,x}]_{\e\wedge \sigma^\varepsilon(x)} > \lambda_\e \big )
 \]
 is asymptotically exponentially negligible as $\e \ra 0$.  
 \paragraph{1. Additive case: } We start with the additive case $G(x, z) = z$ such that 
 \[
  [\xX^{\e,x}]_t = \e^2 \int_0^t \int_{\sqrt{\e} |\ln(\e)|^{q}} |z|^{2} N^{\frac{1}{\e}}(ds,dz). 
 \]
 First we observe that due to Hypothesis E.1 
 $\iota(\e) :=  \int_{|z|\lqq \sqrt{\e} |\ln\e|^{q}} |z|^2 \nu(dz)\ra 0$ as $\e\ra 0$.
 Campbell's formula yields  
 \begin{align*}
 \bar\EE\Big[\exp\big(\kappa [\xX^{\e,x}]_{\e}\big)\Big] 
 &= \exp\Big(\e \int_{|z|\lqq \sqrt{\e} |\ln\e|^{q}} (\exp(\kappa \e^2 |z|^2)-1)\Big) \frac{1}{\e} \nu(dz)\Big). 
 \end{align*}
 If we choose $\kappa = \kappa_\e$ of order 
 \[
 \kappa_\e = \frac{|\ln(\e)|^p}{\e^{2}} 
 \]
 we obtain that as upper bound for the exponent 
 \[
 \lim_{\e \ra 0} \sup_{|z|\lqq \sqrt{\e}|\ln\e|^{q}} \kappa_\e \e^2 |z|^2= \frac{ \e^2 |\ln(\e)|^p(\sqrt{\e} |\ln\e|^{q})^2}{\e^{2}} = \lim_{\e \ra 0} \e|\ln \e|^{2q+p}  = 0 
 \]
 and for all $r\in (0, 1)$ we have by Taylor's theorem $\exp(r)-1 \lqq 2 r$.
Therefore  
 \begin{align*}
 \int_{|z|\lqq \sqrt{\e} |\ln\e|^{q}} \big(\exp(\kappa_\e \e^2 |z|^2)-1\big) \nu(dz) 
 &\lqq  2 \kappa_\e \e^2 \int_{|z|\lqq \e^{1/2}|\ln\e|^{q}}  |z|^2 \nu(dz) \\
 &= 2 |\ln(\e)|^p \int_{|z|\lqq \e^{1/2}|\ln\e|^{q}}  |z|^2 \nu(dz)
 \end{align*}
and Hypothesis E.3 implies 
 \begin{align*}
\limsup_{\e \ra 0} \int_{|z|\lqq \sqrt{\e} |\ln\e|^{q}} \big(\exp(\kappa_\e \e^2 |z|^2)-1\big) \nu(dz) 
 &\lqq \limsup_{\e \ra 0}  2 |\ln(\e)|^p \int_{|z|\lqq \e^{1/2}|\ln\e|^{q}}  |z|^2 \nu(dz) < \infty. 
 \end{align*}
This implies for $\iota(\e) = \int_{|z|\lqq \e^{1/2}|\ln\e|^{q}}  |z|^2 \nu(dz)$ the upper bound 
 \begin{align*}
\limsup_{\e\ra 0} \bar\EE\Big[\exp\big(\kappa_\e [\xX^{\e,x}]_{\e}\big)\Big] 
 &\lqq \limsup_{\e\ra 0} \exp\Big(2|\ln(\e)|^q \iota(\e)\Big) \lqq \limsup_{\e\ra 0}  \exp\Big(2|\ln(\e)|^q \iota(\e)\Big) = E_\infty < \infty. 
 \end{align*}
Hence 
 \begin{align*}
 \bar\PP([\xX^{\e,x}]_{\e}> \la_\e)
 & \lqq \exp(-\kappa_\e \frac{\e}{\sqrt{|\ln(\e)|^p}}) \exp\Big(2|\ln(\e)|^{q} \iota(\e)\Big)\\
 & = \exp(-\frac{|\ln(\e)|^p}{\e} \frac{\e}{\sqrt{|\ln(\e)|}^{p}}) \exp\Big(2|\ln(\e)|^{q} \iota(\e)\Big)
 \lqq \exp(- \frac{\sqrt{|\ln(\e)|^p}}{\e})  E_\infty,
 \end{align*}
which is asymptotically exponentially negligible as $\e \ra 0$. Summing up we have obtained
\begin{align*}
&\e \ln \sup_{x \in D} \bar \PP ( \displaystyle \sup_{t\in [0,\e]} 
|X_t^{ \e,x} -x| \gqq \rho) \\
&\lqq \e \ln \sup_{x \in D} \Big(\bar \PP ( \displaystyle \sup_{t\in [0,\e]} 
|\xX_t^{ \e,x} -\EE[\xX_t^{ \e,x}]| \gqq \rho) + 2 C e^{-\frac{|\ln(\e)|}{\e}}\Big)\\
&\lqq \e \ln \sup_{x \in D} \Big(\bar \PP ( \displaystyle \sup_{t\in [0,\e]} 
|\xX_t^{ \e,x} -\EE[\xX_t^{ \e,x}]| \gqq \rho~|~[\xX^{\e, x}]_\e\lqq \frac{\e}{\sqrt{|\ln(\e)|^p}}) 
+ \PP([\xX^{\e, x}]_\e> \frac{\e}{\sqrt{|\ln(\e)|^p}}) \\
&\qquad + 2 C \exp(-\frac{|\ln(\e)|^q}{\e})\Big)\\
&\lqq \e \ln \sup_{x \in D} \bigg(2 \exp \Big ( - \frac{1}{8} \frac{\rho^2 \sqrt{|\ln(\e)|^q}}{\e}\Big ) 
+ \exp\Big(- \frac{\sqrt{|\ln(\e)|^p}}{\e}\Big) E_\infty + 2 C \exp\Big(-\Gamma\frac{|\ln(\e)|^q}{\e}\Big)\bigg)\lra -\infty 
\end{align*}
as $\e \ra 0$. 
\bigskip
\paragraph{2. Multiplicative case: } 
For any $\e>0$ it follows
 \begin{align*}
 [\xX^{\e,x}- \bar \EE[\xX^{\e,x}]]_{\e} = [\xX^{\e,x}]_{\e}\lqq
 \e^2 \int_0^{\e} \int_{0 < |z|\lqq\sqrt{\e} |\ln(\e)|^q} 
 |G(\xX^{\e, x}_{s-}, z)|^{2} N^{\frac{1}{\e}}(ds,dz).
 \end{align*}
 By Hypothesis \ref{condition: on the multiplicative coefficient} and 
 due to the nonnegativity of the increments in the preceding quadratic variation 
 we have  
 \begin{align*}
 [\xX^{\e,x}]_{t}\lqq
 \e^2 2L^2\int_0^{t}\int_{0 < |z|\lqq\sqrt{\e} |\ln(\e)|^q} (1+|\xX^{\e, x}_{s-}|^2)  
 |z|^{2} N^{\frac{1}{\e}}(ds,dz) 
 \end{align*}
and
 \begin{align*}
 [\xX^{\e,x}]_{t\wedge \sigma^\varepsilon(x)}
 &\lqq \e^2 2L^2  \int_0^{t\wedge \sigma^\varepsilon(x)}\int_{0 < |z|\lqq\sqrt{\e} |\ln(\e)|^q} L(1+|\xX^{\e, x}_{s-}|^2)  
 |z|^{2} N^{\frac{1}{\e}}(ds,dz)\\
 &\lqq 2L^2 (1+\diam(D)^2))\e^2 \int_0^{t}\int_{0 < |z|\lqq\sqrt{\e} |\ln(\e)|^q} |z|^{2} N^{\frac{1}{\e}}(ds,dz)\\
  &\lqq 2L^2 (1+\diam(D)^2)) \e^2 \int_0^{t}\int_{0 < |z|\lqq\sqrt{\e} |\ln(\e)|^q} |z|^{2} N^{\frac{1}{\e}}(ds,dz).
 \end{align*}
 As a consequence we may repeat the same reasoning of 
 the additive case for $[\xX^{\e,x}]_{\e\wedge \sigma^\varepsilon(x)}$ instead of $[\xX^{\e,x}]_{\e}$, that is, 
 for sufficiently small $\e$ we have the asymptotically exponentially negligible estimate 
 \begin{align*}
 \bar\PP([\xX^{\e,x}]_{\e\wedge \sigma^\varepsilon(x)}> \frac{\e}{\sqrt{|\ln(\e)|^p}})
 & \lqq e^{-\kappa_\e \frac{\e}{\sqrt{|\ln(\e)|^p}}} \exp\Big(4L^2 (1+\diam(D)^2) |\ln(\e)|^{p} \iota(\e)\Big)\\
 &= e^{- \frac{\sqrt{|\ln(\e)|^p}}{\e}}  (E_\infty)^{2L^2(1+\diam(D)^2)}. 
 \end{align*}
 Hence the virtually identical reasoning as in the additive case yields the desired result. 
\end{proof}

\medskip 
\begin{lemma} \label{chpt2: lemma third lemma lower bound first exit time} Let $F \subset D^c$ closed. Then
\begin{align*}
\lim_{\rho \rightarrow 0} \limsup_{\e \rightarrow 0} \e \ln \sup_{x \in B_{\rho}(0)} \bar \PP (X_{\vt^x_\rho}^{ \e, x} \in F)  
\lqq - \inf_{z \in F} V(0,z).
\end{align*}
\end{lemma} 

\begin{proof}
Fix $\delta >0$ and $V_F (\delta) :=  (\displaystyle \inf_{z \in F} V(0,z) - \delta) \wedge \frac{1}{\delta}$. 
By definition of $V$ we have 
\begin{align*}
V(x,z)\lqq V(x,y) + V(y,z) \,\,\, \mbox{ for all }\, x,y,z \in \mathbb{R}^d.
\end{align*}
By Lemma \ref{prop: continuity pps cost functional} there is $\rho_0>0$ 
such that for $\rho \in (0, \rho_0]$ we have 
\begin{align*}
\displaystyle \inf_{z \in F, y \in \bar B_{\rho}(0)} V(y,z) \gqq \displaystyle \inf_{z \in F} V(0,z) 
- \displaystyle \sup_{y \in \bar B_{\rho}(0)} V(0,y) \gqq V_F (\delta).
\end{align*}
Lemma \ref{chpt2: lemma second lemma upper bound} 
provides a constant $T >0$ such that for any $\rho \in (0, \rho_0]$
\[  
\displaystyle \limsup_{\e \rightarrow 0} \e \ln \displaystyle \sup_{y \in \bar B_{\rho}(0)}  \bar \PP (\vt^\e_{\rho} (y) > T) 
< - V_F (\delta).
\]
We consider the following subset of $ \DD([0,T], \mathbb{R}^d)$
\begin{align*}
\mathcal{A} := \{  \varphi \in \DD ([0,T], \mathbb{R}^d) ~|~ \varphi(s) \in F \quad \text{for some } s \in [0, T]\}.
\end{align*} 
We have that $\aA$ is a closed set of $\DD([0,T], \RR^d)$ for the Skorokhod topology. 
For a proof we refer to Lemma \ref{claim: second topological claim} in the Appendix. 
Corollary \ref{chpt2: corol Ldp uniform in compact sets of initial states}
implies that there exists $\rho_0>0$ such that for $0 < \rho < \rho_0$,
\begin{align*}
\displaystyle \limsup_{\e \rightarrow 0} \e \ln \displaystyle \sup_{y \in \bar B_{\rho}(0)}  \bar \PP (X^{ \e, y} \in \mathcal{A}) 
& \lqq - \displaystyle \inf_{y \in \bar B_{\rho}(0)} \displaystyle \inf_{\varphi \in \mathcal{A}} \mathbb{J}_{y, T} (\varphi) 
\lqq - \displaystyle \inf_{y \in \bar B_{\rho}(0), z \in F} V(y,z) \lqq - V_F (\delta).
\end{align*}
Finally we have
\begin{align*}
&\displaystyle \limsup_{\e \rightarrow 0+} \e \ln \displaystyle \sup_{x\in \bar B_{\rho}(0)} 
\bar \PP (X^{ \e, x}_{\vt^x_{\rho}} \in F) \\
& \lqq \displaystyle \limsup_{\e \rightarrow 0+} \e \ln \displaystyle \sup_{x\in \bar  B_{\rho}(0)}  \bar \PP (\vt^\e_\rho(x) < \infty) \\
&   \lqq  \displaystyle \limsup_{\e \rightarrow 0+} \e \ln \displaystyle \sup_{x\in \bar  B_{\rho}(0)} 
\bar \PP ( \{ \vt^\e_\rho (x) >  T \} \cup \{ \vt^\e_\rho (x) \lqq   T \} )\\ 
&  \lqq \displaystyle \limsup_{\e \rightarrow 0+} \e \ln \Big (  \displaystyle \sup_{y\in \bar B_{\rho}(0)} 
\bar  \PP ( \vt^\e_\rho (y)> T ) + \displaystyle \sup_{y\in \bar B_{\rho}(0)} \bar  \PP (X^{ \e, y} \in \mathcal{A}) \Big ) 
 \lqq - V_F(\delta).
\end{align*}
Sending $\delta \ra 0$ finishes the proof.
\end{proof}

\begin{theorem} \label{chpt2: thm lower bound first exit time}
Let $\delta >0, x \in D$. Then we have
\begin{align*}
\lim_{\e \rightarrow 0 } \bar \PP (\sigma^{\e} (x)  \lqq e^{\frac{\bar{V}- \delta}{\e}}) &= 0 \text{ and } \nonumber \\
 \liminf_{\varepsilon \ra 0} \varepsilon \ln \bar \EE [\sigma^\varepsilon(x)] &\gqq \bar V - \delta.
\end{align*}
\end{theorem}
\begin{proof}
 The proof is organized in three consecutive steps. We start with the case $\bar V>0$. 
\paragraph{Step 1.} 
Due to Hypothesis \ref{condition: domain} there is $\rho' > 0$ such that $\bar B_{  \rho'} (0) \subset D$ and $\displaystyle \sup_{x\in B_{\rho'}(0)} \langle b(x), n(x)\rangle < 0$ and let $\rho>0$ such that $B_\rho(0) \subset B_{\rho'}(0)$. We define recursively for any $x \in D$
\begin{align} \label{eq: markov chain for the first exit times} 
\zeta^x_0 &:= 0 \non  \quad \text{and for any } k \in \NN\\
\vt_{k, \rho}^x &:= \inf \{ t \gqq \zeta^x_k ~|~ X^{ \e, x}_t \in \bar B_{\rho}(0) \cup D^c \}, \non \\
\zeta^x_{k+1} & := 
\begin{cases} \infty, \quad &\text{if } X^{ \e, x}_{\vt_{k, \rho}^x} \in D^c, \\
\inf \{  t \gqq \vt_{k, \rho}^x ~|~ X^{ \e, x}_t \in \bar B^c_{\rho'} (0) \}, \quad &\text{if } X^{ \e, x}_{\vt_{k, \rho}^x} \in \bar B_\rho(0).\\
\end{cases} 
\end{align}
By construction $(\zeta^x_k)_{k \in \NN}$ and $(\vt_{k, \rho}^x)_{k \in \NN}$ we have $\bar \PP$-a.s. for all $k \in \NN$ 
\begin{align*}
  \zeta^x_{k} \lqq \vartheta^x_{k, \rho} \lqq \zeta^x_{k+1} \lqq \vartheta^x_{k+1, \rho}.
\end{align*}
Since $\rho'> \rho$ we have that $\zeta^x_{k+1} > \vartheta^x_{k}$ if $X^{\varepsilon,x}_{\vartheta^x_{k}} \in \bar B_\rho(0)$. Hence $(\vt^x_{k, \rho})_{k \in \NN}$ is an increasing sequence of $(\fF_t)-$stopping times. Since
the process $(X^{\e,x}_t)_{t \gqq 0}$ has the strong Markov property with respect to $(\fF_t)_{t \gqq 0}$ it follows that 
$(X^{ \e, x}_{\vt_{k, \rho}^x})_{k \in \NN}$ is a Markov chain and  $\sigma^\varepsilon(x) = \vartheta^x_{\ell,\rho}$ for some (random) $\ell \in \NN$ with the convention 
$X^{\e, x}_{\vt_{\ell, \rho}^x} := X^{ \e, x}_{\sigma^{\e} (x)}$ if $\vt_{\ell, \rho}^x = \infty$.  
\begin{claim} \label{claim: inclusion sets for conclusion first exit time problem}
For any $x \in D$, $\e >0$, $T>0$ and $k \in \NN$ arbitrary it follows 
\begin{align} \label{eq: inclusion sets for conclusion first exit time problem}
\{ \sigma^{\e} (x)  \lqq  T \} \subset \{ \sigma^{\e} (x)  =  \vt_0^{x, \rho} \} \cup 
\bigcup_{m=1}^{k} \{ \sigma^{\e} (x) = \vt_{m, \rho}^x \} \cup \{\zeta_m^x - \vt^x_{m-1, \rho} \lqq 
T \wedge\sigma^\varepsilon(x)\}.  
\end{align}
\end{claim}
\begin{proof}
Fix $k \in \NN$ such that $\sigma^\varepsilon(x)> \vartheta^x_{k,\rho}$ and $\zeta^x_m - \vartheta^x_{m-1,\rho} >
T\wedge \sigma^\varepsilon(x)$ for every $m \in \{1, \dots,k \}$. It follows
\begin{align*}
\sigma^\varepsilon(x) > \vt^x_{k, \rho} = \sum_{m=1}^{k} (\vt^x_{m, \rho}- \vt^x_{m-1, \rho}) + \vt^x_{0, \rho} 
 \gqq \sum_{m=1}^k (\zeta^x_{m}- \vt^x_{m-1, \rho})
 > k (T\wedge \sigma^\varepsilon(x)).
\end{align*}
Contraposition of the preceding inclusion of events and eliminating redundancies yield 
\begin{align*} 
\{ \sigma^{\e} (x)  \lqq k T \} &\subset \{ \sigma^{\e} (x)  \lqq \vt_0^{x, \rho} \} \cup 
\bigcup_{m=1}^{k} \{ \sigma^{\e} (x) \lqq \vt_{m, \rho}^x \} \cup \{\zeta_m^x - \vt^x_{m-1, \rho} \lqq 
T\wedge\sigma^\varepsilon(x) \} \\
& =  \{ \sigma^{\e} (x)  =  \vt_0^{x, \rho} \} \cup 
\bigcup_{m=1}^{k} \{ \sigma^{\e} (x) = \vt_{m, \rho}^x \} \cup \{\zeta_m^x - \vt^x_{m-1, \rho} 
\lqq T\wedge \sigma^\varepsilon(x) \}.
\end{align*}
\end{proof}
\no Fix $\delta >0$. We set $k(\e) := \left \lfloor{ \frac{1}{T(\e)} e^{\frac{\bar V - \delta}{\e}}}\right \rfloor+1$
for $T(\e) = \e$. For any $x \in D$ this yields
\begin{align*} 
\bar \PP (\sigma^\e(x) \lqq e^{\frac{\bar V - \delta}{\e}}) \lqq \bar \PP (\sigma^\e(x) \lqq k(\varepsilon) 
(T(\e) \wedge \sigma^\varepsilon(x))).
\end{align*}
The inclusion of events \eqref{eq: inclusion sets for conclusion first exit time problem} implies that
\begin{align} \label{eq: second inclusion events0}
\bar \PP(\sigma^{\e} (x)  \lqq k(\e) T ) \lqq \bar \PP (  \sigma^{\e} (x)  =  \vt_0^{x, \rho} ) +  
\sum_{m=1}^{k(\e)} \Big (\bar \PP( \sigma^{\e} (x) = \vt_{m, \rho}^x ) + \bar \PP( \zeta_m^x - \vt^x_{m-1, \rho} \lqq 
T(\e) \wedge \sigma^\varepsilon(x)) ) \Big ).  
\end{align}

\paragraph{Step 2.} Using Lemma \ref{chpt2: lemma third lemma lower bound first exit time}
there exists $\e_0 > 0$ and $\rho>0$ such that for $\e \in (0, \e_0]$ we have
\begin{align} \label{eq: first estimate to the end --1}
\displaystyle \limsup_{\e \rightarrow 0} \e \ln \displaystyle \sup_{x\in \bar  B_{\rho}(0)} 
\bar \PP (X^{ \e, x}_{\vt^x_{\rho}}  \in D^c) \lqq - \bar{V} + \dfrac{\delta}{2}.
\end{align}

\no Let $0 < \rho < \rho_0$ satisfy Step 1 such that $0 <\rho' - \rho < \rho_0$ for some $\rho_0>0$ fixed below, 
$x  \in D$ and $m \gqq 1$. Since $\sigma^\varepsilon(x) < \infty$ $\bar \PP$-a.s. and due to the strong Markov property we have
\begin{align} \label{eq: first estimate to the end-0}
\sup_{x \in D} \bar  \PP (\sigma^{\e} (x) = \vt^x_{m, \rho}) \nonumber
& \lqq \sup_{x \in D} \bar \PP (X^{ \e, x}_{\vt^x_{m, \rho}} \in D^c; \vartheta^x_{m, \rho} < \infty) \nonumber  \\
& \lqq  \sup_{x \in D} \bar \PP (X^{ \e, x}_{\vt^x_{m-1, \rho}} \in \bar B_{\rho}(0); X^{ \e, x}_{\vt^x_{m, \rho}} \in D^c) \lqq  \sup_{y \in \bar B_{\rho}(0)} \bar  \PP (X^{\e}_{\vt^y_{\rho}} \in D^c).
\end{align}
\no Lemma \ref{chpt2: lemma third lemma lower bound first exit time} yields for any $m \gqq 1$ and $\delta>0$ that there exists $\rho_0>0$ such that for any $\rho < \rho_0$ we have
\begin{align} \label{eq: first estimate to the end-0.5}
\limsup_{\varepsilon \ra 0} \varepsilon \ln \bar \PP(\sigma^\varepsilon(x)= \vartheta^x_{m, \rho}) \lqq -\bar V + \frac{\delta}{2}. 
\end{align}
\no For the constant $c= \bar V - \frac{\delta}{2}>0$ and $\rho_0 >0$ such that for $ \rho' -\rho \lqq \rho_0$ 
Lemma \ref{chpt2: lemma second lemma lower bound first exit time} implies 
\begin{align} \label{eq: first estimate to the end-1}
\sup_{y \in D} \bar \PP (\zeta_k^y - \vt^y_{k-1, \rho} \lqq T(\e) \wedge \sigma^\varepsilon(x))
& \lqq  \displaystyle  \sup_{y \in B_{\rho}(0)}  \bar \PP \Big ( \displaystyle  \sup_{t\in [0, T(\e) \wedge \sigma^\varepsilon(x)]} |X^{ \e, y}_t - y| \gqq \rho'-\rho \Big ) \non \\
&\lqq   \displaystyle  \sup_{y \in D}  \bar \PP \Big ( \displaystyle  \sup_{t \in [0, T(\e) \wedge \sigma^\varepsilon(x)]} |X^{\e, y}_t - y| \gqq \rho' - \rho\Big ).
\end{align}
Therefore, for the constant $c= \bar V - \frac{\delta}{2}>0$ fixed above we obtain 
\begin{align} \label{eq: first estimate to the end-2}
\displaystyle \limsup_{\varepsilon \ra 0} \varepsilon \ln  \displaystyle \sup_{y \in D} \bar \PP(\displaystyle \sup_{t \in [0,T(\e) \wedge \sigma^\varepsilon(x)]} |X^{\varepsilon,y}-y| \gqq \rho' -\rho) < - \Big (\bar V-  \frac{\delta}{2} \Big).
\end{align}
Hence, combining \eqref{eq: second inclusion events0}-\eqref{eq: first estimate to the end-2} yields for any $k \in \NN$ and $x \in D$ that 
\begin{align} \label{eq: first estimate to the end-3}
&\bar \PP (\sigma^{\e} (x)  \lqq k(\varepsilon) (T(\e) \wedge \sigma^\varepsilon(x)))  \nonumber\\
&\lqq \bar  \PP (\sigma^{\e} (x)  = \vt_{0, \rho}^x ) + \sum_{m=1}^{k(\varepsilon)} \Big ( \bar \PP (\sigma^{\e} (x) 
= \vt_{m,\rho}^x)+ \bar \PP (\zeta_{m}^x - \vt_{m-1, \rho}^x \lqq T(\e) \wedge \sigma^\varepsilon(x)) \Big ) \nonumber\\
& \lqq  \bar \PP (\sigma^{\e} (x) = \vt_{0, \rho}^x) + 2 k(\varepsilon) e^{- \frac{\bar V - \frac{\delta}{2}}{\varepsilon}}.
\end{align}
Due to Lemma \ref{chpt2: lemma first lemma in lower bound first exit time} 
we have for all $x \in D$ the desired result
\begin{align} \label{eq: to use in the first exit time asymptotics full2}
\bar \PP ( \sigma^{\e} (x)  \lqq e^{\frac{\bar{V}- \delta}{\e}}) & \lqq \bar  \PP  (\sigma^{\e} (x)  \lqq k(\varepsilon) T(\e)) \lqq \bar \PP (X^{ \e, x}_{\vt_{0,\rho}^x} \notin \bar{B}_{\rho} (0)) + \frac{2}{T(\e)} e^{-\frac{\delta}{2 \e}} 
\rightarrow 0 \quad \text{ as } \e \ra 0.
\end{align}
 Chebyshev's inequality implies for $\varepsilon>0$ sufficiently small that
\begin{align*}
 \bar \EE [\sigma^\e(x)] & \gqq e^{\frac{\bar V- \delta}{\e}} \bar \PP (\sigma^\e(x) \gqq e^{\frac{\bar V - \delta}{\e}}) 
 \gqq \frac{1}{2} e^{\frac{\bar V - \delta}{\e}}.
 \end{align*}
 This finishes the proof for the case $\bar V>0$. 
 \medskip
 \noindent \paragraph*{Step 3. } We treat the  case $\bar{V}= 0$.  
 Let $\delta > 0$ and $x \in D$. Choose $\rho > 0$ such that $\bar{B}_{ \rho}(0) \subset D$. 
 Assume $c >0$. By the strong Markov property
 we can choose $\e_0 \in (0, 1]$ such that for $\e \in (0, \e_0]$ 
\begin{align} \label{eq: lower bound final}
\bar \PP (\sigma^{\e} (x) > e^{-\frac{\delta}{\e}}) \gqq \bar \PP (X^{ \e, x}_{\vt^x_{\rho}} \in \bar{B}_\rho (0)) 
\inf_{y \in B_{ \rho}(0)}  \bar \PP (\sup_{t \in [0, \e\wedge \sigma^{\e} (x)]} |X^{ \e, y}- y| \lqq \rho).
\end{align}
 Lemma \ref{chpt2: lemma first lemma in lower bound first exit time} and 
 Lemma \ref{chpt2: lemma second lemma lower bound first exit time}  imply that the  right-hand side of (\ref{eq: lower bound final}) converges to $1$ as $\e \rightarrow 0$.  This finishes the proof.
\end{proof}
\paragraph*{The exit location in Theorem \ref{thm: first exit time}:}
\begin{rem} \label{rem: exit location}
The proof of the statement 2. of Theorem \ref{thm: first exit time} goes along the same line of 
reasoning as in the Brownian case and is extensively documented in the literature in different settings. We refer the reader for example to Theorem 4.2.4 in \cite{FW98} to a a more general setting 
for the deterministic dynamical system \eqref{eq: determinsitic ode} with an additive Brownian perturbation and to Theorem 5.7.11 in \cite{DZ98} for a multiplicative Brownian perturbation of \eqref{eq: determinsitic ode}. 
Our result is obtained with analogous arguments used to prove the second statement of Theorem 2.4.6 in \cite{Siegert} (pp. 88-90). 
Therefore we omit the proof and refer the reader to \cite{Siegert}.
\end{rem}

\bigskip
\section{Appendix}
\subsection{Preliminaries for the proof of the LDP} \label{sec: preliminaries LDP}
\subsubsection{Proof of Lemma \ref{lemma: integrability controls}} \label{section: lemma int controls}
Let $\nu \in \mathfrak{M}$ satisfy Hypothesis \ref{condition:generalized statement- the measure nu}.
\paragraph{Step 1.} We start with the proof of  \eqref{eq: integrability controls- first order}.  Let $g \in \mathfrak{S}^M$.
We have the decomposition
  \begin{align} \label{chpt2: eq int controls eq 0}
  \int_0^T \int_{\RR^d \backslash \{0\}} |z|^2 g(s,z) \nu(dz) ds \lqq   \int_0^T \int_{0<|z|\lqq 1} |z|^2 g(s,z) \nu(dz) ds +  \int_0^T \int_{|z|> 1} |z|^2 g(s,z) \nu(dz) ds.
  \end{align}
The first integral on the right hand side of (\ref{chpt2: eq int controls eq 0}) is estimated as follows. 
Young's inequality reads for any $a, b \gqq 0$ that $ab \lqq e^a + b \ln b -b$. This implies that
\begin{align} \label{chpt2: eq int controls eq 0.5}
\int_0^T \int_{0<|z|\lqq 1} |z|^2 g(s,z) \nu(dz) ds &\lqq \int_0^T \int_{0<|z|\lqq 1} |z|^2 (e + \ell (g(s,z))) \nu(dz) ds \nonumber \\
& \lqq e T c_\nu^2 + \int_0^T \int_{0< |z|\lqq 1} |z|^2 \ell(g(s,z))\nu(dz) ds \lqq e T c_\nu^2 + M  < \infty
\end{align}
since $\nu$ is a L\'{e}vy measure ($c_\nu^2 := \int_{0<|z|\lqq 1} |z|^2 < \infty$).
For the second integral in the right hand side of (\ref{chpt2: eq int controls eq 0}) we estimate
\begin{align} \label{chpt2: eq int controls 1}
\int_0^T \int_{|z|\gqq 1} |z|^2 g(s,z) \nu(dz)ds &\lqq T \Big ( \int_{|z|\gqq 1} e^{\Gamma |z|^2} \nu(dz) \Big ) + \frac{1}{\Gamma} \int_0^T \int_{|z|\gqq 1} \ell (g(s,z))\nu(dz)dz  < \infty
\end{align}
since $\nu$ satisfies the integrability assumption (\ref{eq: integrability cond measure}) 
and $g \in \mathfrak{S}^M$. Combining (\ref{chpt2: eq int controls eq 0}), (\ref{chpt2: eq int controls eq 0.5}) 
and  (\ref{chpt2: eq int controls 1}) yields (\ref{eq: integrability controls- first order}).
\medskip
 \paragraph{Step 2.} 
 We fix $M>0$, $g \in \mathfrak{S}^M$ and $I \subset [0,T]$ a measurable set. Remark 3.3. in \cite{BCD13} 
 states for any $\beta>0$ that 
\begin{align} \label{eq: lemma int controls rmk}
&|x-1| \lqq c_1(\beta) \ell(x)\quad \text{if } |x-1| \gqq \beta \text{ and } |x-1|^2 \lqq c_2(\beta) \ell(x) \quad  \text{if } |x-1| \lqq \beta,
\end{align} 
 for some $c_1(\beta), c_2(\beta) >0$ where $c_1(\beta) \ra 0$ as $\beta \ra 0$.
 Let $\beta>0$ and consider 
 the measurable set $$E_\beta:=\{ (s,z) \in I \times \RR^d \backslash \{0\} ~|~ |g(s,z)-1| \lqq \beta \}.$$
We apply the following version of Young's inequality: $ab \lqq e^{\sigma a} + \frac{1}{\sigma} \ell(b)$  for any $a,b>0$ and $\sigma\gqq 1$, 
and obtain for any $\beta>0$ and  $\sigma \gqq 1$ 
 \begin{align} \label{eq: lemma int controls final}
 &\int_I \int_{\RR^d \backslash \{0\}} |z||g(s,z)-1| \nu(dz) ds 
 \lqq \int_{E_\beta} |z| |g(s,z)-1| \nu(dz)ds + \int_{E_\beta^c} |z||g(s,z)-1| \nu(dz)ds \nonumber \\
 &\quad \lqq \Big (\int_I \int_{\RR^d \backslash \{0\}} |z|^2 \nu(dz)ds \Big )^{\frac{1}{2}} \sqrt{c_2(\beta)} \sqrt{M} + \int_{E_\beta^c \cap \{ |z|>1 \}} |z|(g(s,z)-1) \nu(dz)ds \nonumber\\
 &\qquad + \int_{E_\beta^c \cap \{0<|z|<1\}} |z|(g(s,z)-1) \nu(dz) ds \nonumber \\
 &\quad \lqq  \sqrt{|I| C_1 c_2(\beta) M} 
 +  |I| (C_2+C_3)+ \frac{M}{\sigma} + c_1(\beta)M.
 \end{align}
 In the preceding estimate we used
\begin{align*}
  C_1:= \displaystyle \int_{\RR^d \backslash \{0\}} |z|^2 \nu(dz), \quad
 C_2:= \displaystyle \int_{|z|>1} e^{\sigma|z|} \nu(dz) \quad \text{and } C_3 := \int_{|z|>1}|z| \nu(dz),
\end{align*} 
 all of which are finite since $\nu$ is a L\'{e}vy measure on $(\RR^d \backslash \{0\}, \bB(\RR^d \backslash \{0\}))$ 
 and satisfies (\ref{eq: integrability cond measure}). Choosing $I= [0,T]$ in (\ref{eq: lemma int controls final}) 
 proves (\ref{eq: integrability controls- first order1}). 
 \medskip
\paragraph{Step 3.}
In order to prove (\ref{eq: integrability controls- limit zero}) let us fix $\delta'>0$ arbitrary 
and  $t, t' \in [0,T]$ such that $|t - t' | < \delta$ with $\delta>0$ fixed below. 
Let $\beta>0$ be sufficiently large such that $c_1(\beta)M < \frac{\delta'}{4}$ 
and $\sigma> \frac{4M}{\delta'}$. Finally fix $\delta>0$ such that
$\delta < \frac{\delta'}{4 (C_2+C_3)} \wedge \frac{(\delta')^2}{16(C_1 c_2(\beta)M)^2}$. 
For the choice of $\beta, \sigma \gqq 1$ and $\delta>0$ as above, one has that (\ref{eq: lemma int controls final}) implies
\begin{align*}
\int_t^{t'} \int_{\RR^d \backslash \{0\}} |z||g(s,z)-1| \nu(dz) ds  < \delta'.
\end{align*}
This finishes the proof of (\ref{eq: integrability controls- limit zero}).
 \begin{flushright} $\square$ \end{flushright}

\subsubsection{Proof of Proposition \ref{proposition: priori estimate controlled process}} \label{section: Bernstein}
For convenience of notation we drop the dependence of $\tilde X^{\varepsilon,x}$ on $x$. 
For every $\varepsilon>0$ let $\rR(\varepsilon)>0$ such that $\rR(\varepsilon) \ra \infty$ 
and $\varepsilon \rR^2(\varepsilon) \ra 0$ as $\varepsilon \ra 0$, for example 
$\rR(\varepsilon):= |\ln \varepsilon|$, $\varepsilon>0$. By definition of $\tilde \tau^\e_{\rR(\e)}$ in (\ref{eq: the stopping time}) it follows
\begin{align*}
\bar \PP \Big ( \displaystyle \sup_{t\in [0, T]} |\tilde X^\e_t| > \rR(\e) \Big ) 
\lqq \bar \PP \Big ( \displaystyle \sup_{t \in [0, \tilde \tau^\e_{\rR(\e)} \wedge T]} |\tilde X^\e_t| > \rR(\e) \Big ).
 \end{align*}
We observe that for every $\e>0$ the process $(\tilde X^\e_t)_{t \in [0,T]}$ 
 is a locally square integrable martingale. Therefore we use the Bernstein-type inequality given by Theorem 3.3 of 
 \cite{DZ01} and infer for some parameter $\lambda= \lambda_\e>0$ that is fixed below
\begin{align} \label{eq: localization prop-bernstein}
\bar \PP \Big ( \displaystyle \sup_{t\in [0,  \tilde  \tau^\e_{\rR(\e)} \wedge T]} |\tilde X^\e_t| > \rR(\e) \Big ) 
& \lqq \bar \PP \Big ( \displaystyle \sup_{t \in [0, \tilde  \tau^\e_{\rR(\e)} \wedge T]} |\tilde X^\e_t| > \rR(\e) , \; [\tilde X^\e]_{\tilde \tau^\e_{\rR(\e)} \wedge T} \lqq \lambda \Big ) + \bar \PP \Big ( [\tilde X^\e]_{\tilde \tau^\e_{\rR(\e)} \wedge T} > \lambda \Big ) \nonumber \\
 & \lqq 2 \exp \Big (-\frac{1}{2}\frac{\rR^2(\e)}{\lambda} \Big ) + \bar \PP \Big ( [\tilde X^\e]_{\tilde \tau^\e_{\rR(\e)}} > \lambda \Big ).
 \end{align}
 For every $\e>0$ the quadratic variation of the process 
 $(\tilde X^\e)_{t \in [0,T]}$ is given for every $t \in [0,T]$ by
 \begin{align*}
 [\tilde X^\e]_t = \e^2 \int_0^t \int_{\RR^d \backslash \{0\}} |G(\tilde X^{\e}_{s-})|^2 |z|^2 N^{\frac{\varphi_\e}{\e}}(ds,dz).
\end{align*}
Due to Hypothesis \ref{condition: on the multiplicative coefficient} and 
Chebyshev's inequality the second probability of the last term of (\ref{eq: localization prop-bernstein}) 
is estimated for $\e \in (0, \e_0]$ with $\e_0>0$ small enough as follows
\begin{align} \label{eq: localization prop- quadratic variation}
\bar \PP \Big ( [\tilde X^\e]_{ \tilde \tau^\e_{\rR(\e)} \wedge T} > \lambda \Big ) 
& \lqq  \bar \PP \Big ( \e^2 \int_0^{\tilde \tau^\e_{\rR(\e)} \wedge T } 
\int_{\RR^d \backslash \{0\}} |G(\tilde X^{\e}_{s-},z)|^2 N^{\frac{\varphi_\e}{\e}}(ds, dz) > \lambda \Big ) \nonumber \\
& \lqq \bar \PP \Big ( 2 L^2 \e^2(1+ \rR^2(\varepsilon)) \int_0^T \int_{\RR^d \backslash \{0\}} |z|^{2} \varphi_\e(s,z) \nu(dz) ds  > \lambda \Big )  \nonumber\\
& \lqq \frac{2 L^2 \bar C \varepsilon(1+ \rR^2(\varepsilon))}{\lambda}, 
\end{align}
where $\bar C := \displaystyle \sup_{g \in \mathfrak{S}^M} \int_0^T \int_{\RR^d \backslash \{0\}} |z|^2 g(s,z) \nu(dz)ds < \infty$ 
by (\ref{eq: integrability controls- first order}) of Lemma \ref{lemma: integrability controls}. 
Set $\lambda= \lambda_\e= \rR(\e)$, $\varepsilon>0$. Combining (\ref{eq: localization prop-bernstein}) 
and (\ref{eq: localization prop- quadratic variation}) yields some $C>0$ and  $\varepsilon_0>0$ such that $\varepsilon< \varepsilon_0$ implies
\begin{align*}
\bar \PP \Big ( \displaystyle \sup_{t\in [0, T]} |\tilde X^\e_s| > \rR(\e)\Big ) \lqq 2 e^{-\frac{1}{2} \rR(\e)} + C \e \rR(\e), 
\end{align*}
which converges to $0$ as $\varepsilon \ra 0$.

 \begin{flushright} $\square$ \end{flushright}
 \bigskip
 
\subsubsection{Proof of Proposition \ref{proposition: a priori bound}} \label{section: a priori bound}
For every $\varepsilon>0$ let $\rR(\varepsilon)>0$ fixed as in the statement of Proposition \ref{proposition: priori estimate controlled process}, 
i.e. such that $\rR(\varepsilon) \ra \infty$ and $\varepsilon \rR^2(\varepsilon) \ra 0$ as $\varepsilon \ra 0$ and 
$\tilde \tau^\varepsilon_{\rR(\varepsilon)}$ in (\ref{eq: the stopping time}). 
We drop the dependence on $x \in \RR^d$ wherever possible without confusion. 
Ito's formula, Hypotheses \ref{condition: det dynamical system}, \ref{condition: the measure nu} 
and \ref{condition: on the multiplicative coefficient} and the inequality (\ref{eq: integrability controls- first order}) 
imply on the event $\{T < \tilde \tau^\varepsilon_{\rR(\varepsilon)} \}$ for any $t \in [0,T]$, $\varepsilon<\varepsilon_1$ 
with $\varepsilon_1>0$ sufficiently small and $\bar \PP$-a.s. that the following holds
\begin{align} \label{eq: prop a priori bound Ito}
|\tilde X^\varepsilon_t|^2 &\lqq |x|^2 + 2 \int_0^t \int_{\RR^d \backslash \{0\}} \langle G(\tilde X^\varepsilon_s)z, \tilde X^\varepsilon_{s} \rangle (\varphi_\varepsilon(s,z)-1) \nu(dz)ds \nonumber \\
&\quad  + M^\varepsilon_1(t) + M^\varepsilon_2(t) + \varepsilon \int_0^t \int_{\RR^d \backslash \{0\}} |G(\tilde X^\varepsilon_s) z|^2 \varphi^\varepsilon(s,z)\nu(dz)ds  \nonumber \\
& \lqq  |x|^2 + 2 L^2 C_0 + 2 L^2 \int_0^t |\tilde X^\varepsilon_s|^2 \Theta^\varepsilon(s) ds +  M^\varepsilon_1(t) + M^\varepsilon_2(t)+ 2 \varepsilon L^2 (1+ \rR^2(\varepsilon))C_1  \nonumber\\
& \lqq 2 |x|^2 + 2 L^2 C_0 + 2 L^2 \int_0^t |\tilde X^\varepsilon_s|^2 \Theta^\varepsilon(s) ds +  M^\varepsilon_1(t) + M^\varepsilon_2(t),
\end{align}
where for any $t \in [0,T]$ and $\varepsilon>0$ we denote the processes 
\begin{align*}
\begin{cases}
M^\varepsilon_1(s)&:= \displaystyle \int_0^t \int_{\RR^d \backslash \{0\}} \varepsilon^2|G(\tilde X^\varepsilon_{s-})z|^2 \tilde N^{\frac{1}{\varepsilon} \varphi^\varepsilon}(ds,dz), \\
M^\varepsilon_2(s)&:= \displaystyle \int_0^t \int_{\RR^d \backslash \{0\}} 2 \varepsilon \langle G(\tilde X^{\varepsilon}_{s-})z , \tilde X^\varepsilon_{s-} \rangle \tilde N^{\frac{1}{\varepsilon} \varphi^\varepsilon}(ds,dz),
\end{cases}
\end{align*}
and the constants 
\begin{align*}
\begin{cases}
C_0 &:= \displaystyle \sup_{g \in \mathfrak{S}^M} \int_0^T \int_{\RR^d \backslash \{0\}} |z| |g(s,z)-1| \nu(dz)ds < \infty, \quad \text{due to (\ref{eq: integrability controls- first order1})}, \\
C_1 &:= \displaystyle \sup_{g \in \mathfrak{S}^M} \int_0^T \int_{\RR^d \backslash \{0\}} |z|^2 g(s,z) \nu(dz)ds< \infty, \quad \text{due to (\ref{eq: integrability controls- first order})}.
\end{cases}
\end{align*}
In addition, $\Theta^\varepsilon(s):= \int_{\RR^d \backslash \{ 0\}} |z||\varphi^\varepsilon(s,z)-1| \nu(dz) ds$ is such that
$\int_0^T \Theta^\varepsilon(s,z)ds < \infty$ due to (\ref{eq: integrability controls- first order1}). 
Gronwall's lemma yields a constant $C_2>0$ such that
for every $\varepsilon>0$ small enough, the event $\{ T >  \tilde \tau^\varepsilon_{\rR(\varepsilon)}\}$ implies
\begin{align} \label{eq: a priori estimate-Gronwall}
\displaystyle \sup_{t \in [0,T]} |\tilde X^\varepsilon_t|^2 \lqq C_2 e^{2L^2 \int_0^T \Theta^\varepsilon(s)ds} \Big ( 1 + \displaystyle \sup_{t \in [0,T]} |M^\varepsilon_1(t)| + \displaystyle \sup_{t\in [0, T]} |M^\varepsilon_2(t)| \Big ). 
\end{align}
The Burkholder-Davis-Gundy and the Jensen inequalities yield some $C_3>0$ such that 
\begin{align} \label{eq: a priori estimate-BDG1}
\bar \EE \Big [ \displaystyle \sup_{t \in [0, \tilde \tau^\varepsilon_{\rR(\varepsilon)}]} |M^\varepsilon_1(s)|\Big ] & \lqq C_3 \bar \EE \Big [ \Big ( \int_0^{\tilde \tau^\varepsilon_{\rR(\varepsilon)}} \int_{\RR^d \backslash \{0\}} \varepsilon^4 |G(\tilde X^\varepsilon_{s-},z)|^2 N^{\frac{1}{\varepsilon} \varphi^\varepsilon}(ds,dz) \Big )^{\frac{1}{2}}\Big ] \nonumber \\
& \lqq C_3 \varepsilon^2 \rR(\varepsilon) \sqrt{\bar \EE \Big [ \int_0^T \int_{\RR^d \backslash \{0\}} |z|^2 N^{\frac{1}{\varepsilon} \varphi^\varepsilon}(ds,dz) \Big ]} \nonumber \\
& \lqq C_3 \frac{\varepsilon^2 \rR^2(\varepsilon)}{\sqrt{\varepsilon}} \sqrt{C_1} \ra 0 \quad \text{ as }\varepsilon\ra 0. 
\end{align}
Analogously it is shown that
\begin{align} \label{eq: a priori estimate: BDG2}
\bar \EE \Big [ \displaystyle \sup_{ t \in [0,\tilde \tau^\varepsilon_{\rR(\varepsilon)}]} |M^\varepsilon_2(t)| \Big ] \ra 0, \quad \text{as } \varepsilon \ra 0.
\end{align}
Hence from (\ref{eq: a priori estimate-Gronwall}), (\ref{eq: a priori estimate-BDG1}) 
and (\ref{eq: a priori estimate: BDG2}) it follows for some $C_4>0$ and every $\varepsilon>0$ small enough that
\begin{align*}
\bar \EE \Big [ \displaystyle \sup_{t \in [0,\tilde \tau^\varepsilon_{\rR(\varepsilon)}]} |\tilde X^\varepsilon_t|^2 \Big ] & \lqq C_4 \Big ( 1 + \bar \EE \Big [ \displaystyle \sup_{t \in [0,\tilde \tau^\varepsilon_{\rR(\varepsilon)}]} |M^\varepsilon_1(t)|\Big ] + \bar \EE \Big [ \displaystyle \sup_{t \in [0,\tilde \tau^\varepsilon_{\rR(\varepsilon)}]} |M^\varepsilon_2(t)|\Big ]\Big )  < \infty.
\end{align*}
This finishes the proof.
\begin{flushright} $\square$ \end{flushright}
\bigskip

\subsection{Proof of Lemma \ref{prop: continuity pps cost functional}}  \label{subsec: continuity cost function}
The statements (\ref{eq: pp cont1}) and (\ref{eq: pp cont2}) are a consequence of the following fact. 
For any fixed $M>0$ and $g \in \mathfrak{S}^M$ Proposition \ref{condition: potential} ensures for any $\rho>0$ some $\xi(\rho)>0$ such that
$\xi(\rho) \ra 0$ as $\rho \ra 0$ and $\Phi \in C([0,\xi(\rho)], \RR^d)$ 
solving  \eqref{eq: cond on potential: controllability cond C1}.
Since for the function $\ell(b)= b \ln b - b +1$, $b \gqq 0$ we have 
\begin{align*}
\int_0^T \int_{\RR^d \backslash \{0\}} \ell (g(s,z)) \nu(dz)ds \lqq M,
\end{align*}
the  monotone convergence theorem yields
\begin{align*}
\displaystyle \lim_{\rho \ra 0} \int_0^{\xi(\rho)} \int_{\RR^d \backslash \{0\}} \ell (g(s,z)) \nu(dz) ds 
= \int_0^{T} \int_{\RR^d \backslash \{0\}}  \displaystyle \lim_{\rho \ra 0} \textbf{1}_{[0, \xi(\rho)]}(s) \ell(g(s,z)) \nu(dz) ds = 0. 
\end{align*}
Hence for any $\delta>0$ there is $\rho_0>0$ small enough such that $\rho \in (0, \rho_0]$ implies 
$V(x,y,t) \lqq \eE_{\xi(\rho)} (g) \lqq \delta$.
\begin{flushright} $\square$ \end{flushright}
\bigskip
\subsection{Topological properties of the Skorokhod space used in Section \ref{sec: first exit} } \label{subsec: topology}
\begin{lem} \label{claim: first topological claim}
Given $t >0$, $D \subset \RR^d$ a bounded domain, $\rho>0$ and the sets
\begin{align*}
\mathcal{G}_t &:= \Big  \{ \Phi \in  \DD ([0,t], \mathbb{R}^d)~|~ \Phi(s) \in 
\overline{D \backslash B_\rho (0)} \quad \mbox{for all } s \in [0,t] \Big \}, \\
 \tilde \gG_t &:=  \Big \{  \Phi \in  \DD ([0,t], \mathbb{R}^d)~|~ \Phi(s) \in \overline{D \backslash B_\rho (0)} 
 \mbox{ for all  } s \in [0,t] \\ 
 &\qquad  \mbox{ except in a countable number of points} \Big  \},
\end{align*}
we have  that $\tilde \gG_t = \gG_t$ and $\tilde \gG_t$ is a closed set in $\DD([0,t], \RR^d)$ 
with respect to the Skorokhod topology. 
\end{lem}

\begin{proof}\hfill
\paragraph*{Step 1: } We prove that $\tilde \gG_t$ is closed in $\DD([0,t], \RR^d)$ with respect 
to the Skorokhod topology. Let $(\Phi_n)_{n \in \NN} \subset \tilde \gG_t$ such that $d_{J_1} (\Phi_n, \Phi) \ra 0$ as $n \ra \infty$ 
for some $\Phi \in \DD([0,t], \RR^d)$.
Let $(s_k)_{k \in \NN}$ be the countable  set of discontinuities of $\Phi$. For each $ n \in \NN$ 
we denote $(t^n_k)_{k \in \NN}$ the countable set such that
$$\Phi_n (s) \in \overline{D \backslash B_{\rho} (0)} \quad \mbox{ for all } s \in [0,t] \backslash  \{t^n_k\}_{k \in \NN}.$$
For all $s \in [0,t] \backslash \big ( \bigcup_{n=1}^{\infty} \{t^n_k\}_{k \in \NN} \cup \{s_k\}_{k \in \NN} \big )$ 
it is a standard property of c\`{a}dl\`{a}g functions (see \cite{Billingsley}, p.112) that $$\Phi_n (s) \ra \Phi(s) \quad \mbox{as } n \ra \infty.$$
Since $ \overline{D \backslash B_{\rho} (0)}$ is a compact set of $\RR^d$, $\Phi(s) \in  \overline{D \backslash B_{\rho} (0)}$, 
which concludes the proof that $\tilde \gG_t$ is closed in $(\DD([0,t], \RR^d), J_1)$. 

\paragraph{Step 2: }  We prove that $\tilde \gG_t = \gG_t$. 
The inclusion $\tilde \gG_t \supset \gG_t$ is obvious. 
Let $\Phi \in \tilde \gG_t$. If there exists $s \in [0,t]$ such that $\Phi(s) \notin \overline{ D \backslash B_\rho(0)}$, by right-continuity of $\Phi$, there exists $\delta>0$ such that
\begin{align*}
\Phi [s, s + \delta) \subset (\bar D)^c   \cup B_{\rho} (0)
\end{align*}
which violates $\Phi \in \tilde \gG_t$. 

\end{proof}

\begin{lem} \label{claim: second topological claim}
For any closed set $F\subset \bB(\RR^d)$ and $t>0$ 
we consider the following subset of $ \DD([0,t], \mathbb{R}^d)$
\begin{align*}
\mathcal{A} := \{  \varphi \in \DD ([0,t], \mathbb{R}^d) ~|~ \varphi(s) \in F \quad \text{for some } s \in [0,t]\}.
\end{align*} 
Then $\aA$ is closed in $\DD([0,t], \RR^d)$ with respect to the Skorokhod topology.
\end{lem}

\begin{proof}
Let $(\varphi_n)_{n \in \NN}$ be a sequence of elements of $\aA$ and $\varphi \in \DD([0, t], \RR^d)$ 
such that $d_{J_1}(\varphi_n, \varphi) \ra 0$ as $n \ra \infty$. 
For every $n \in \NN$ let $s_n \in [0, t]$ such that $\varphi_n(s_n) \in F$. By right continuity 
of $\varphi_n$ there exists $\delta_n>0$ such that $\varphi_n([s_n, s_n + \delta_n)) \subset F$. 
For every $n \in \NN$ we denote $I_n:= [s_n, s_n + \delta_n)$. For every $n \in \NN$ let $\{t_n^k\}_{k \in \NN}$ 
be the set of discontinuities of $\varphi$ in $I_n$. Since for every $n \in \NN$ 
$\varphi_n$ and $\varphi$ are c\`{a}dl\`{a}g functions we have
\begin{align*}
\varphi_n(r) \ra \varphi(r) \quad \text{ for all } r \in \bigcup_{n \in \NN} (I_n \backslash \{t_n^k\}_{k \in \NN}).
\end{align*}
Since $F$ is a closed subset of $\RR^d$ $\varphi(r) \in F$ for all $ r \in \bigcup_{n \in \NN} (I_n \backslash \{t_n^k\}_{k \in \NN})$. 
This proves that $\varphi \in \aA$ and that $\aA$ is closed in $\DD([0,T], \RR^d)$ with respect to the Skorokhod topology. 
\end{proof}

\paragraph*{Acknowledgments} The first author thanks Peter Imkeller (HU Berlin) and Sylvie Roelly (U. Potsdam) for the fruitful discussions on the subject. He also acknowledges the financial support from the project MASH(51099907) during his stay at U. Potsdam and from the FAPESP grant number 2018/06531-1 at UNICAMP-Campinas (SP).
The second author would like to thank the School of Sciences at Universidad de los Andes for FAPA funding and MINCIENCIAS for travel funding in the framework of the Stic AMSUD 2019 Project "Stochastic analysis of non-Markovian phenomena". 
Both authors greatly acknowledge financial and infrastructure support by the DFG-funded International Research Training Group (IRTG) 1740 Berlin- S\~ao Paulo: 
Dynamical Phenomena in Complex Networks: Fundamentals and Applications hosted at Humboldt-Universit\"at Berlin.

\end{document}